\documentclass[12pt,reqno]{amsart}
\usepackage[utf8]{inputenc}
\usepackage{graphicx}
\usepackage{xcolor}
\usepackage{hyperref}
\usepackage{bbm}
\usepackage{amssymb}
\usepackage{array}
\usepackage{fullpage}
\newcommand{\cC}{\mathcal{C}}
\newcommand{\cD}{\mathcal{D}}
\newcommand{\cZ}{Z}
\newcommand{\ZZ}{\mathbbm{Z}}
\newcommand{\CC}{\mathbbm{C}}

\DeclareMathOperator{\Hom}{Hom}
\DeclareMathOperator{\FPdim}{FPdim}
\DeclareMathOperator{\Irr}{Irr}
\DeclareMathOperator{\Aut}{Aut}
\DeclareMathOperator{\End}{End}
\DeclareMathOperator{\Fun}{Fun}
\DeclareMathOperator{\Fix}{Fix}

\DeclareMathOperator{\Vect}{Vec}
\DeclareMathOperator{\Tr}{Tr}
\DeclareMathOperator{\ev}{ev}
\DeclareMathOperator{\id}{id}
\DeclareMathOperator{\Dim}{Dim}
\DeclareMathOperator{\dimCC}{dim_{\CC}}
\renewcommand{\dim}{\Dim}

\newcommand{\tu}{\mathbbm{1}}
\newcommand{\br}{\mathrm{br}}
\newtheorem{thm}{Theorem}[section]
\newtheorem{prop}[thm]{Proposition}
\newtheorem{cor}[thm]{Corollary}
\theoremstyle{definition}
\newtheorem{defn}[thm]{Definition}
\newtheorem{lem}[thm]{Lemma}
\newtheorem{rem}[thm]{Remark}
\newcommand{\rev}[1]{{#1}^\mathrm{rev}}
\newcommand{\slot}{\,\cdot\,}

\usepackage[alphabetic,initials]{amsrefs}

\usepackage{graphicx,tikz}
\usetikzlibrary{arrows,backgrounds,scopes}
\usetikzlibrary{calc,cd}
\newcommand{\tikzmath}[2][0.42]
{\vcenter{\hbox{\begin{tikzpicture}[scale=#1] #2\end{tikzpicture}}}
}
\tikzset{coupon/.style={rectangle,rounded corners=1.5pt,draw,fill=white,inner sep=1.5,minimum size=12pt}}
\newcommand{\mydotw}[1]{\begin{scope}[shift={#1}] \fill[shift only,white] (0,0) circle (1.5pt); \draw[shift only,thick]  (0,0) circle (1.5pt);   \end{scope}}

\title{Computing fusion rules for spherical $G$-extensions of fusion categories }
\author{Marcel Bischoff}
\address{Department of Mathematics, Ohio University,
Athens, OH 45701, USA}
\email{bischoff@ohio.edu}
\email{marcel@localconformal.net}
\thanks{M.B. is supported by NSF Grant DMS-1700192/1821162}
\author{Corey Jones}
\address{{Department of Mathematics, The Ohio State University, 
Columbus, OH 43210, USA}}
\email{jones.6457@osu.edu}
\thanks{C.J. is supported by NSF Grant DMS-1901082}
\date{\today}
\begin{document}
\begin{abstract}
A $G$-graded extension of a fusion category $\mathcal{C}$ yields a categorical action of $G$ on the center $Z(\mathcal C)$. If the extension admits a spherical structure, we provide a method for recovering its fusion rules in terms of the action. We then apply this to find closed formulas for the fusion rules of extensions of some group theoretical categories and of cyclic permutation crossed extensions of modular categories.
\end{abstract}
\maketitle
\tableofcontents

\section{Introduction}

The theory of fusion categories has found significant applications in the study of two dimensional quantum physics, most notably in conformal field theory \cite{MoSe1990, FuRuSc2002, BiKaLoRe2014-2, HuLep2013} and topological phases of matter \cite{Kit06, NSSFD2008, Wang2010}. In both these contexts modular tensor categories appears as important invariants of physical models. If the model has a group $G$ of global symmetries, one obtains a $G$-crossed braided fusion category which is a $G$-extension of the original invariant \cite{Ki2002, Mg2005, BaBoChWa2014}. This makes understanding of $G$-extensions of fusion categories of fundamental importance for the study of two dimensional symmetry enriched physical systems.

The theory developed by Etingof, Nikshych, and Ostrik \cite{EtNiOs2010} provides the basic tools to construct and classify $G$-extensions of fusion categories. They show that every extension of $\mathcal{C}$ can be constructed from a categorical action, i.e. a monoidal functor $M:\underline{G}\rightarrow \underline{\operatorname{Aut}}^{\mathrm{br}}_{\otimes}(Z(\mathcal{C}))$, provided a certain cohomology class $o_{4}(M)\in H^{4}(G, \CC^{\times})$ is trivial. In the case this obstruction vanishes, the possible extensions associated with this action form a torsor over the group $H^{3}(G, \CC^{\times})$, but all the extensions have the same fusion rules. This tells us that in principle the fusion rules of the extension can be computed from the initial categorical action $M$. However in practice this problem is usually difficult. Naively following the proofs of the above statements from \cite{EtNiOs2010} requires the computation of the data for the associated bimodule categories, their relative Deligne products, and the bimodule functors used to define the monoidal product on the extension (see Section \ref{extension theory}). The amount of computation required to find the data in these intermediate steps quickly becomes infeasible as the rank of the fusion category grows.

In this paper, we provide a method for computing the fusion rules of an extension in an elementary way from a detailed knowledge of $M$ and $Z(\mathcal{C})$. Our approach bypasses the computation of the associated bimodule categories and their data. It allows for the derivation of closed form expressions of fusion rules for families of extensions in some general cases. The key observation in our approach is that the fusion rules can be recovered from the composition and convolution operations on the space of endomorphisms of the canonical Lagrangian algebra $I(\tu)\in Z(\cC)$ (see Corollary \ref{recovfusionrulesthm}). This may be viewed as a direct generalization of character theory for the representation category of a finite group .

We now give an outline of how this works. It is well known (for spherical fusion categories) that $\operatorname{End}(I(\tu))\cong K_{0}(\cC)\otimes_{\mathbbm{Z}} \mathbbm{C}$ as associative complex algebras. However, given $\operatorname{End}(I(\tu))$ as an abstract algebra, to find the fusion rules we need more information. We also need to identify the canonical basis elements $\{[X]\}_{X\in \text{Irr}(\cC)}$ (or perhaps some scaled version of them) so that we can recover the fusion rules by examining the coefficients under multiplication. 

Luckily there is an additional operation on $\operatorname{End}(I(\tu))$ that allows us to recover the (appropriately scaled) canonical basis in a canonical way. For any commutative Frobenius algebra $A$ in a braided fusion category, there are two associative binary operations on the vector space $\operatorname{End}(A)$. The first is the usual composition of morphisms, which in general is noncommutative. The second is the \textit{convolution operation} $\ast$ (see equation \ref{defconvolutionproduct}, \cite{Bi2016, BiDa2018}). This operation makes the vector space $\operatorname{End}(A)$ into a commutative algebra in the usual sense. If $A=I(\tu)\in Z(\cC)$ is the canonical Lagrangian algebra, we show the minimal idempotents $e_{X}$ with respect to $\ast$ are in bijective correspondence with equivalence classes simple objects $X\in \text{Irr}(\cC)$ (see equation \ref{eYdefinition}). Since the algebra $(\operatorname{End}(I(\tu), \ast)$ is commutative and semi-simple, the minimal idempotents give a canonical basis for the space $\operatorname{End}(I(\tu))$. We then show that $e_{X}\circ e_{Y}=\sum_{Z\in \text{Irr}(\cC)}\frac{d_{X}d_{Y}}{d_{Z}} N^{Z}_{XY} e_{Z}$ (where $d_{X}$ indicates (any) spherical dimension function, see Proposition \ref{fusioncomposition}). Thus while the basis $\{e_{X}\}_{X\in \text{Irr}(\cC)}$ is not quite \textit{the} canonical basis $\{[X]\}_{X\in \Irr(\cC)}$ for the fusion ring described above, the quantity $\frac{d_{X}d_{Y}}{d_{Z}}$ is independent of the spherical structure, and thus we can recover the fusion rules $N^{Z}_{XY}$ by examining the numbers $C^{Z}_{XY}$ defined by $(e_{X}\circ e_{Y})\ast e_{Z}=C^{Z}_{XY}e_{Z}$, and renormalizing (see equation \ref{fusioncoeff}).

Therefore, for a $G$-extension $\cC\subseteq \cD$, the fusion rules of $\cD$ can be determined by computing the composition and convolution products on the endomorphisms of the canonical Lagrangian algebra in $\cZ(\cD)$. By \cite{GeNaNi2009}, the latter category is equivalent to the equivariantization of the $G$-crossed braided relative center $\cZ_{\cC}(\cD)$. Here the $G$ action restricts to $M$ on the trivially graded component $Z(\cC)\subseteq Z_{\cC}(\cD)$. Furthermore, the canonical Lagrangian algebra for $\cD$ lives in the subcategory $Z(\cC)^{G}\subseteq Z(\cD)$. In particular, if $I_{G}:Z(\cC)\rightarrow Z(\cC)^{G}$ is the adjoint of the forgetful functor $F_{G}:Z(\cC)^{G}\rightarrow Z(\cC)$, then $I_{\cD}(\tu)\cong I_{G}(I_{\cC}(\tu))\in Z(\cC)^{G}\subseteq Z(\cD)$. Using the adjunction between $I_{G}$ and $F_{G}$, we can compute the triple $(\operatorname{End}(I_{\cD}(\mathbbm{1}),\ast,\circ)$ terms of the data of $\cZ(\cC)$ and the category action $M$. (Section \ref{Gexts}). We can then recover the fusion rules as described above.

A subtlety is that for the numbers we produce to actually be the fusion rules of the extension, we need to assume that $\cD$ admits a spherical structure, though we do not need to explicitly choose one (see Remark \ref{Nospherical}). Unfortunately, the extension theory of \cite{EtNiOs2010} has not been developed to take spherical structures into account, hence it is not clear a-priori if the $G$-extensions constructed from a given categorical action admit spherical structures. However, if we make the mild assumption that $\cC$ is \textit{pseudo-unitary}, then any extension is automatically so and hence our hypothesis is satisfied (see Proposition \ref{pseudounitaryext}). In this case our results apply to any $G$-extension, without additional hypothesis.

As a first application, we utilize our method to give general formulas for fusion rules of $G$-extensions of $\operatorname{Vec}(\hat{A}\times A, q)_{L}$ where $A$ is an abelian group, $q$ is the canonical hyperbolic quadratic form, and $L\le \hat{A}\times A$ is a Lagrangian subgroup. The extensions depend on an initial braided categorical action on $\operatorname{Vec}(\hat{A}\times A, q)\cong Z(\operatorname{Vec}(\hat{A}\times A, q)_{L})$ (see \textbf{Theorem \ref{pointedcase}}). Here $\operatorname{Vec}(\hat{A}\times A, q)_{L}$ denotes the fusion category of modules of the group algebra object associated to the Lagrangian subgroup. Note that the categories $\operatorname{Vec}(\hat{A}\times A, q)_{L}$ are precisely those which are Morita equivalent to $\operatorname{Vec}(A)$.

We then focus on the case when $\cC$ is modular, and the categorical action can be factored $\underline{G}\rightarrow \underline{\operatorname{Aut}}^{\mathrm{br}}_{\otimes}(\cC)\rightarrow \underline{\operatorname{Aut}}^{\mathrm{br}}_{\otimes}(\cZ(\cC))$, where the second functor acts on the right factor in $\cZ(\cC)\cong \cC\boxtimes \rev{\cC}$. If a corresponding extension exists, it has the additional structure of a \textit{$G$-crossed braided} extension of $\cC$. 
These are the extensions which naturally appear both in conformal field theory \cite{Mg2005} and topological phases of matter \cite{BaBoChWa2014}, hence are of the greatest interest in applications. 
In this case, the nice form of the Lagrangian algebra and of the action allows us to describe the convolution product in a general way. 

The examples of this type we consider are permutation actions on $\cC^{\boxtimes n}$. 
These have long been of interest to physicists in the context of rational conformal field theory \cites{BoHaSch,Ban2002,LoXu2004,KaLoXu2005,Mg2005} as an intermediate step in the study of permutation orbifold theories. More recently, permutation extensions have been of interest in the theory of topological phases under the guise of ``genons'' for their potential in quantum computing applications \cites{BaJiQi2013,BaBoChWa2014}.

Permutation crossed extensions have also come to attract the attention of mathematicians. They have been studied from the point of view of modular functors \cite{BarSchweig}. From an algebraic viewpoint, the $o_{4}$ obstruction for permutation actions was shown to vanish in \cite{GaJo2018}, hence these extensions always exist. 
They have been studied in the $\ZZ/2\ZZ$ case (\cites{BarSchweig,BaFRS}, \cites{EdJoPl2018,Pa2018}). 
Very recently, Delaney has given an algorithm for computing the fusion rules of general permutation extensions using the concept of bare defects \cite{D2019}. Here we will use our method to give a closed formula for the fusion rules in the case of \textit{maximal cyclic} permutation extensions (see \textbf{Theorem \ref{fusioncyclic}}).
Our formulas for the fusion rules involve the dimensions of vectors spaces assigned by the modular functor derived from $\cC$ to surfaces with field insertions. While our approach for fusion rules is different from \cite{D2019}, we have verified that their algorithm produces the same numbers as our formula in several examples.

The outline of the paper is as follows. The preliminary section briefly collects some facts about fusion categories, modular categories, equivariantizations, and extension theory that will be used in the paper. In Section \ref{recovfusionrules}, we demonstrate how to reconstruct the based fusion ring from the canonical Lagrangian algebra in $\cZ(\cC)$ and apply this to $G$-extensions as described above. Finally, we turn to the case of $G$-crossed extensions of modular categories, giving explicit examples of the computation of fusion rules for $G$-extensions from a given categorical action. We an include an appendix with a list of fusion rules for the $\ZZ/4\ZZ$ cyclic permutation extension of the modular category $\text{Fib}^{\boxtimes 4}$

\subsection{Acknowledgments}
The authors would like to thank Colleen Delaney, Cain Edie-Michell, Dave Penneys and Julia Plavnik for very useful discussions and comments on an early draft. We also thank Colleen Delaney for sharing an early draft of \cite{D2019} with us and for coordinating arXiv submissions. Marcel Bischoff was supported by NSF grant DMS-1700192/1821162. Corey Jones was supported by NSF Grant DMS-1901082.

\section{Preliminaries}

Recall a \textit{fusion category} is a $\mathbbm{k}$-linear, finitely semi-simple, rigid, monoidal category with simple unit \cite[Chapter 4]{EtGeNiOs2015}. In this paper, we assume $\mathbbm{k}=\mathbbm{C}$. The semi-simplicity gives us well behaved \textit{fusion rules}, described by the non-negative integers $N^{Z}_{XY}=\dim(\cC(X\otimes Y, Z))$ for $X,Y,Z\in \Irr(\cC)$. Here and for the rest of the paper, we use $\Irr(\cC)$ to denote a fixed choice of representative for each equivalence classes of simple object in $\cC$. If $\cC$ is any category, we use here and throughout the paper the notation $\cC(X,Y):=\operatorname{Hom}(X,Y)$. We typically use $f\circ g$ to represent composition of morphisms.

For fusion categories, there are several notions of dimension that are important to consider. First, there is a unique function $\FPdim\colon\Irr(\cC)\rightarrow \mathbbm{R}^{+}$ such that
$\FPdim(\tu)=1$ and $\FPdim(X)\FPdim(Y)=\sum_{Z\in \Irr(\cC)}N^{Z}_{XY}\FPdim(Z)$ called the \textit{Frobenius-Perron dimension}, \cite[Section 3.3]{EtGeNiOs2015}. This dimension depends only on the based ring $K_{0}(\cC)$, and is insensitive to the details of the categorification.


The second notion of dimension depends on a choice of \textit{spherical structure}. This is a monoidal natural isomorphism from the identity to the double dual functor $X\mapsto \bar{\bar{X}}$ such that the associated left and right pivotal traces are equal \cite[Chapter 4.7]{EtGeNiOs2015}. A spherical structure gives us a single, well-defined spherical trace for every object $X\in \cC$, $\Tr_{X}\colon \cC(X,X) \rightarrow \mathbbm{C}$. We can then define the spherical dimension function $d\colon \Irr(\cC)\rightarrow \mathbbm{R}_{\neq 0}$, $d_{X}:=\Tr_{X}(1_{X})$ which satisfies $d_{X}d_{Y}=\sum_{Z\in \Irr(\cC)}N^{Z}_{XY}d_{Z}$. Spherical structures also allow us to make use of the spherical graphical calculus, which we use freely \cites{BW,Tu1994,BaKi2001,Se2011}. 
It is an open question whether every fusion category admits a spherical structure \cite{EtNiOs2005}.

There is a third important notion of dimension in fusion categories. Let $X$ be a simple object in a fusion category, and let $\bar{X}$ be a (two-sided) dual object. Choose arbitrary evaluation and coevaluation morphisms $R_{X}\in \cC(\tu, \bar{X}\otimes X), \bar{R}_{X}\in \cC(\tu, X\otimes \bar{X}), R^{*}_{X}\in \cC( \bar{X}\otimes X,\tu), \bar{R}^{*}_{X}\in \cC( \tu, X\otimes \bar{X}) $ satisfying the duality equations. Then the quantity

\begin{equation}\label{paireddimension}
d_{\{X,\bar{X}\}}1_{\tu}:=(R^{*}_{X}\circ R_{X})(\bar{R}^{*}_{X}\circ \bar{R}_{X})
\end{equation}

\noindent is called the \textit{paired dimension} and does not depend on the choices of $\bar{X}$ or evaluation and coevaluation morphisms. This number is strictly positive \cite{Ba16}, and thus we can define the \textit{fusion dimension} of $X$ as the positive square root
$$
    d^{+}_{X}:=\sqrt{d_{\{X,\bar{X}\}}}\,.
$$
We note that we are using different notation from \cite{Ba16} for the various dimensions. In particular we use $d^{+}_{X}$ for the fusion dimension instead of for the Frobenius-Perron dimension. Like the Frobenius-Perron dimension, the paired dimension and the fusion dimension are intrinsic to a fusion category and do not depend on a choice of additional structure. However, the fusion dimension depends on the associator of the category and cannot be determined by the fusion ring alone.

\begin{defn}
A fusion category is \textit{pseudo-unitary} if $\text{FPdim}(X)^{2}=d_{\{X,\bar{X}\}}$ for all $X\in \Irr(\cC)$.
\end{defn}

We note that our definition has many equivalent formulations (see \cite[Section 8.3, 8.4]{EtNiOs2005}. If $\cC$ is pseudo-unitary, there exists a canonical spherical structure on $\cC$ whose spherical dimensions are the Frobenius-Perron dimensions \cite[Proposition 8.23]{EtNiOs2005}. All unitary fusion categories are pseudo-unitary, and thus in applications (most relevant) to physics all examples are pseudo-unitary.

\subsection{Modular categories}
\label{ssec:ModularCategories}
Recall a \textit{braided fusion category} is a fusion category equipped with a family of natural isomorphisms $\sigma_{X,Y}\colon X\otimes Y\rightarrow Y\otimes X$ satisfying a family of coherences (namely the hexagon axioms). If $\{X\in \Irr(\cC)\ :\ \sigma_{Y,X}\circ \sigma_{X,Y}=1_{X\otimes Y}\ \text{for all}\ Y\in \Irr(\cC)\}=\{\tu\}$, then we say $\cC$ is \textit{non-degenerately braided}, or simply non-degenerate. If $\cC$ is non-degenerate and in addition equipped with a spherical structure, we say $\cC$ is \textit{modular}.

We refer the reader to \cite{BaKi2001} for an overview of modular categories, modular data, and some of their important properties. Here we use the conventions $S_{X,Y}:=\Tr_{X,\bar{Y}}(\sigma_{\bar{Y},X}\circ \sigma_{X,\bar{Y}})$. Furthermore, we use $\dim\cC:=\sum_{X\in \Irr(\cC)} d^{2}_{X}$ which is a positive number independent of the spherical structure. We use $\sqrt{\dim\cC}$ to denote the positive square root. Non-degeneracy of the category $\cC$ is equivalent to the invertibility of the matrix $S$ \cite{Mg2002}.

For modular categories we have the relation 
$$
    S^{2}=\dim\cC \cdot C
$$ 
where $C_{X,Y}=\delta_{X,\bar{Y}}$ is the charge conjugation matrix, hence $S^{-1}_{X,Y}=\frac{1}{\dim\cC} S_{X,\bar{Y}}$.

We write 
$$
    N_{X_1,\ldots,X_n}^{Y_1,\ldots,Y_m}=\dimCC 
    \cC(X_1\otimes \cdots \otimes X_n,Y_1\otimes\cdots\otimes Y_m)
    \,.
$$

\noindent for the generalized fusion coefficients. Note that Frobenius reciprocity for fusion categories gives us $N_{X_1,\ldots,X_n}^{Y_1,\ldots,Y_m}=N_{X_1,\ldots,X_{n-1}}^{Y_1,\ldots,Y_m, \bar{X}_{n}}$.

For $g\in\ZZ_{\geq 0}$
we define the \emph{genus $g$ fusion coefficients} by
\begin{equation}\label{definefusioncoeff}
{}^g N_{X_1,\ldots,X_n}=\sum_{Z_0,\ldots,Z_g\in\Irr(\cC)} N_{X_1,\ldots,X_n}^{Z_0,\ldots,Z_g}N^{\tu}_{Z_0,\ldots,Z_g}\,.
\end{equation}
and use the notation
\begin{equation}
{}^g N^{Y_1,\ldots,Y_m}_{X_1,\ldots,X_n}=
{}^g N_{X_1,\ldots,X_n,\bar Y_1,\ldots,\bar Y_m}\,.
\end{equation}

The terminology is motivated by Remark \ref{rmk:TQFT} (below) and the following \emph{generalized Verlinde formula}. This formula is well known to experts, but we could not find it recorded in the literature, so we provide a proof.
\begin{prop}\label{genverlinde}
    We have 
    \begin{align*}
        {}^gN_{X_1,\ldots,X_n}
        &=
        \sum_{Y\in\Irr(\cC)}
\frac{S_{X_1,Y}}{S_{\tu,Y}}\cdots \frac{S_{X_n,Y}}{S_{\tu,Y}}\left(\frac{\sqrt{\dim \cC}}{S_{\tu,Y}}\right)^{2g-2}
=(\dim\cC)^{g-1}\sum_{Y\in\Irr(\cC)}\frac{S_{X_1,Y}\cdots S_{X_n,Y}}{d_Y^{n+2g-2}}\,.
    \end{align*}
\end{prop}
\begin{proof}
We recall the usual Verlinde formula for modular categories \cites{Tu1994,BaKi2001} 
  $$
    N^{\mathbbm{1}}_{X_1,X_2,X_3}
    =\frac1{\dim \cC} 
    \sum_{Y\in\Irr(\cC)}
    \frac{S_{X_1,Y}S_{X_2,Y} S_{X_3,Y}}
    {S_{\tu,Y}}
  $$
  and the ``sewing'' relation 
  $$
    \sum_{U\in\Irr(\cC)}S_{\bar U,X}S_{U,Z}=\dim\cC\cdot \delta_{X,Z}\,.
  $$
  Then we obtain the $n$-point Verlinde formula
  $$
    N^{\mathbbm{1}}_{X_1,\ldots,X_n}
    =\frac1{\dim \cC} 
    \sum_{Y\in\Irr(\cC)}
    \frac{S_{X_1,Y}\cdots S_{X_n,Y}}
    {S_{\tu,Y}^{n-2}}\,.
  $$
  by induction on $n$.
  
  To see this by semi-simplicity and Frobenius reciprocity, 
  $$
    N^{\mathbbm{1}}_{X_{1}, \dots, X_{n+1}}=\sum_{U}N^{\mathbbm{1}}_{X_{1}, \dots, X_{n-1}, \bar{U}}N^{\mathbbm{1}}_{ U,X_{n}, X_{n+1}}
  $$
  Assuming our formula works for $n$, then
  \begin{align*}
  N^{\mathbbm{1}}_{X_{1}, \dots, X_{n+1}} & =\sum_{U\in \Irr(\cC)}N^{\mathbbm{1}}_{X_{1}, \dots, X_{n-1}, \bar{U}}N^{\mathbbm{1}}_{U,  X_{n}, X_{n+1}}\\
  &=\frac1{(\dim \cC)^2} 
    \sum_{U,V,Y\in\Irr(\cC)}
    \frac{S_{X_1,Y}\cdots S_{X_{n-1},Y}S_{\bar U,Y}}
    {S_{\tu,Y}^{n-2}}\frac{S_{U,V}S_{X_{n},V} S_{X_{n+1},V}}
    {S_{\tu,V}}\\
    &=\frac1{\dim \cC} 
    \sum_{Y\in\Irr(\cC)}
    \frac{S_{X_1,Y}\cdots S_{X_{n+1},Y}}
    {S_{\tu,Y}^{n-1}}
  \end{align*}
    
  Applying sewing to
  \begin{align*}
  &\sum_{Z_0,\ldots,Z_g\in\Irr(\cC)} N_{X_1,\ldots,X_n}^{Z_0,\ldots,Z_g}N^{\mathbbm{1}}_{Z_0,\ldots,Z_g}
  \\&\qquad=\sum_{Z_0,\ldots,Z_g\in\Irr(\cC)} N^{\mathbbm{1}}_{X_1,\ldots,X_n,\bar Z_0,\ldots,\bar Z_g}N^{\mathbbm{1}}_{Z_0,\ldots,Z_g}
  \\&\qquad=\frac{1}{(\dim\cC)^2}
  \sum_{Z_0,\ldots,Z_g,U,V\in\Irr(\cC)}
  \frac{S_{X_1,U}\cdots S_{X_n,U}
  S_{\bar Z_0,U}\cdots S_{\bar Z_g,U}
  }{S_{\tu,U}^{n+g-1}}
  \frac{S_{Z_0,V}\cdots S_{Z_g,V}}
  {S_{\tu,V}^{g-1}}
  \end{align*}
  gives
  the desired result.
\end{proof}

\begin{rem}
\label{rmk:TQFT}
Associated to a modular tensor category $\cC$ is an (anomalous) 3-2-1 topological quantum field theory called the Reshetikhin-Turaev TQFT \cite{BaDoSPVi2015}. To a closed surface with a series of marked points labelled with by simple objects in $\cC$, the TQFT assigns a finite dimensional vector space. The number $M_{\cC}(g; X_1,\cdots, X_n)
^g =N_{X_1,\ldots,X_n}$ is the dimension of this vector space for a genus $g$ surface with marked points $X_{i}\in \cC$. While the vector space itself depends on the ordering and position of the marked points, the dimensions themselves do not. The formula for $M_{\cC}(g; X_1,\cdots, X_n)$ in terms of $S$-matrix appeared in the physics literature \cite[Eq.\ (A.7)]{MoSe1989}.
\end{rem}

The next two subsections review the basics of categorical actions and equivariantization which we will need in the sequel, and the extension theory of \cite{EtNiOs2010}.

\subsection{Equivariantization}\label{Equivariantization}

We now recall some facts about equivariantizations of fusion categories. As a general reference see \cites{EtGeNiOs2015,BuNa}. Let $\cC$ be any fusion category, and $G$ a finite group. Let us first recall some notions related to categorical actions of $G$ on $\cC$:
\begin{itemize}
    \item 
$\underline{G}$ is the monoidal category whose objects are elements of $G$ and the only morphisms are identities. The monoidal product of objects is the product in the group. 
\item
$\underline{\operatorname{Aut}}_{\otimes}(\cC)$ is the monoidal category whose objects are monoidal autoequivalences, and whose morphisms are monoidal natural isomorphisms.  The monoidal product of objects is composition of functors, and the monoidal product of natural isomorphisms is the usual one.
\item
If $\cC$ is braided, then $\underline{\operatorname{Aut}}^{\br}_{\otimes}(\cC)$ is the full monoidal subcategory of $\underline{\operatorname{Aut}}_{\otimes}(\cC)$ whose objects preserve the braiding.
\item
A \textit{categorical action} is a monoidal functor 
$M:\underline{G}\rightarrow \underline{\operatorname{Aut}}_{\otimes}(\cC)$.
\item If $\cC$ is braided, a \textit{braided categorical action} is a monoidal functor
$M:\underline{G}\rightarrow \underline{\operatorname{Aut}}^{\br}_{\otimes}(\cC)$
\end{itemize}

\medskip

\paragraph{\textbf{Notation for categorical actions.}} In what follows below, given a categorical action, we typically denote the functor assigned to $g$ simply by $g(\cdot)$. The tensorator for $g$ is typically indicated by $\rho^{g}_{X,Y}\colon g(X)\otimes g(Y)\rightarrow g(X\otimes Y)$. The tensorator for the categorical action is usually written $\mu_{g,h}=\{\mu^{X}_{g,h}:g(h(X))\rightarrow gh(X)\}_{X\in \cC}.$

\medskip

Given an arbitrary categorical action, recall its \textit{equivariantization} $\cC^{G}$ is defined as follows:

\begin{itemize}
\item
Objects are pairs $(X, \lambda )$ where $\lambda=\{\lambda_{h}: h(X)\cong X\}_{h\in G}$ is a family of natural isomorphisms satisfying

$$\begin{tikzcd}
g(h(X))
\ar[d, "{g(\lambda_{h})}"]
\ar[r,"{\mu^X_{g,h}}"]
&
(gh)(X)
\ar[d,"{\lambda_{gh}}"]
\\
g(X)
\ar[r,"{\lambda_{g}}"]
&
X
\end{tikzcd}$$

\item
Morphisms from $(X, \lambda)$ to $(Y,\delta)$ consist of $f\in\cC(X,Y)$ such that 
$$\delta_{g}g(f)=\lambda_{g}f$$ for all $g\in G$.
\end{itemize}

There is a canonical monoidal structure on this category which makes $\cC^{G}$ a fusion category if $\cC$ is.

Let $F_{G}\colon\cC^{G}\rightarrow \cC$ denote the forgetful functor, which simply forgets the equivariant structure. In this section we provide an explicit realization of an adjoint to this functor. 

\begin{rem}\label{adjoint remark} We briefly remark on our use of the term adjoint. Recall that every linear functor $F$ between semi-simple categories has an adjoint functor, which is both a left and a right adjoint. Since left (and right) adjoints are (respectively) unique up to natural isomorphism, then any left adjoint of $F$ is also a right adjoint, since it must be isomorphic to a two sided adjoint. This is simple a reflection of the fact that in rigid, semi-simple 2-category (in this case, the 2-category of finitely semi-simple categories) left and right duals of 1-morphisms (in this case functors) coincide up to isomorphism. Below we will sometimes need an explicit choice of units and counits for either a left or right adjunction, in order to equip the adjoints of (strong) monoidal functors with oplax and lax monoidal structures respectively.
\end{rem}

We now define a functor $I_{G}: \cC\rightarrow \cC^{G}$ which we call the \textit{induction functor}. On objects, we set 
$$
    I_{G}(X):=(\textstyle\bigoplus_{g\in G} g(X), \eta^{X})
$$
where the equivariant structure $\eta^{X}=\{\eta^{X}_{h}\}_{h\in G}$ is given by

$$
    \eta^{X}_{h}=\bigoplus_{g\in G} \mu^{X}_{h,g} :h(\bigoplus_{g\in G} g(X))=\bigoplus_{g\in G} h(g(X))\rightarrow \bigoplus_{g\in G} hg(X)=\bigoplus_{g\in G}g(X)\,.
$$
For a morphism, $f\in \cC(X,Y)$, we simply define 
$$
    I_{G}(f):=\bigoplus_{g\in G} g(f)\,.
$$

\begin{prop}\label{equivariantadjunction} $I_{G}$ as defined above is adjoint to the forgetful functor $F_{G}: \cC^{G}\rightarrow \cC$ 
\end{prop}

\begin{proof}
By Remark \ref{adjoint remark}, to show $I_{G}$ is an adjoint, it suffices to show $I_{G}$ is either left adjoint or right adjoint to $F_{G}$. However, in this proof we will choose specific left and right adjunctions between $I_{G}$ and $F_{G}$ in order to obtain concrete oplax and lax monoidal structures on $I_{G}$ respectively.

To establish $I_{G}$ as a left adjoint to $F_{G}$ it suffices to build a bijection $\cC^{G}(I_{G}(X), (Y, \lambda))\cong \cC(X, Y)$ natural in both $X$ and $(Y,\lambda)$.

Let $f\in \cC^{G}(I_{G}(X), (Y, \lambda))$. Then as a morphism in $\cC$, we may write $f=\bigoplus_{g\in G} f_{g}$, where $f_{g}: g(X)\rightarrow Y $. Our bijection will be defined by sending $f\mapsto f_{1}$. Since $f$ is equivariant, we see that $f_{gh}=\lambda_{g}\circ g(f_{h})\circ (\mu^{X}_{g,h})^{-1}$ and in particular, $f_{g}=\lambda_{g}\circ g(f_{1})\circ (\mu^{X}_{g,1})^{-1}$ and so $f$ is uniquely determined by $f_1$. Furthermore, for \textit{any} choice of $f^{\prime}\in \cC(X,Y)$, defining $f_{g}=\lambda_{g}\circ g(f^{\prime})\circ (\mu^{X}_{g,1})^{-1}$ yields an equivariant morphism by setting $f=\bigoplus_{g\in G} f_{g}$, establishing the bijection. Naturality in both variables is clear.

The construction of a right adjunction is very similar, we simply reverse the order of the arrows. In particular for the right adjunction, we need to build a family of natural isomorphisms $\cC^{G}((Y,\lambda), I_{G}(X))\cong \cC(Y, X)$. A morphism $f\in \cC^{G}((Y,\lambda), I_{G}(X)) $ is represented as a direct sum $f=\bigoplus_{g\in G}f_{g}$, where $f_{g}\in \cC(Y,X)$. As before, we define our bijection by sending $f\mapsto f_{1}$. As above, since $f$ is equivariant we have for all $g,h\in G$, $f_{gh}=\mu^{X}_{g,h}\circ g(f_{h})\circ \lambda^{-1}_{gh}$. In particular, $f_{g}=\mu^{X}_{g,1}\circ g(f_{1})\circ \lambda^{-1}_{1}$, hence $f$ is completely determined by $f_{1}$. Given any $f^{\prime}\in \cC(Y,X)$, define $f_{g}:=\mu^{X}_{g,1}\circ g(f^{\prime})\circ \lambda^{-1}_{1}$. Then $f=\bigoplus_{g\in G} f_{g}$ gives an equivariant morphism $\cC^{G}((Y,\lambda), I_{G}(X))$ which is clearly inverse to the above assignment.
\end{proof}

By \cite{MR0360749}, since $F_{G}$ is (strong) monoidal, then any right adjoint to $F_{G}$ can be equipped with the structure of an oplax monoidal functor. Similarly, any left adjoint can be equipped with a structure of a lax monoidal functor. These structures depend on the choice of adjunction (i.e. on the units and counits). Since $I_{G}$ is both a right and left adjoint to $F_{G}$, we can use the adjunctions described in the previous proposition to construct both a lax and oplax monoidal structure on $I_{G}$. 

Unpacking the construction in \cite{MR0360749}, we obtain a lax monoidal structure $\nu_{X,Y}\colon I_{G}(X)\otimes I_{G}(Y)\rightarrow I_{G}(X\otimes Y)$ which we can describe explicitly as follows.

Define $$(\nu_{X,Y})^{k}_{g,h}:g(X)\otimes h(Y)\rightarrow k(X\otimes Y)$$ 
by 
$$
    (\nu_{X,Y})^{k}_{g,h}:=\delta_{g,h}\delta_{g,k} \rho^{g}_{X,Y}\,,
$$
where $\rho^{g}_{x,y}$ is the tensorator for the monoidal functor $g$. Then set
$$\nu_{X,Y}:=\bigoplus_{g,h,k} (\nu_{X,Y})^{k}_{g,h}$$
Furthermore, the ``unit'' map of $I_{G}$ is given by a morphism $u\colon \mathbbm{1}_{\cC^{G}}\rightarrow I_{G}(\mathbbm{1}_{\cC})$, 
$$
    u= \oplus_{g\in G} g(1_{\mathbbm{1}})\,. 
$$

\noindent Similarly we can describe the oplax monoidal structure arising from the (left) adjunction in the proof of \ref{equivariantadjunction} as follows. Define 

$$\nu^{\prime}_{X,Y}\colon  I_{G}(X\otimes Y)\rightarrow I_{G}(X)\otimes I_{G}(Y)$$ 

by

$$
\nu^{\prime}_{X,Y}:=\bigoplus_{g,h,k\in G}(\nu^{\prime}_{X,Y})^{g,h}_{k}
$$ 

where 
$$
(\nu^{\prime}_{X,Y})^{g,h}_{k}:=\delta_{k,g}\delta_{k,h} (\rho^{g}_{X,Y})^{-1}\,.
$$

In fact, these lax and oplax structures have an additional compatibility: it's easy to verify that $\nu$ and $\nu^{\prime}$ equip $I_{G}$ with the structure of a ``special Frobenius functor'' \cite{DaPa2008}.

Now let $(A, m, \iota)$ be an algebra object, with multiplication $m\colon A\otimes A\rightarrow A$ and unit $\iota\colon \tu\rightarrow A$. The lax structure $\nu$  on $I_{G}$ allows us to define the algebra $(I_{G}(A),I_{G}(m)\circ \nu_{A,A}, I_{G}(\iota)\circ u)$. If $A$ in addition comes with a coproduct $m^{\prime}\colon A\rightarrow A\otimes A$ making it into a special Frobenius algebra, then $\nu^{\prime}_{A,A}\circ I_{G}(m^{\prime})$ makes $I_{G}(A)$ into a special Frobenius algebra.

\subsection{Extension theory}\label{extension theory}

In this section, we briefly review the extension theory from \cite{EtNiOs2010}. Let $G$ be a finite group. A (faithful) $G$-grading of a fusion category $\cC$ is a decomposition as linear categories $\cC=\bigoplus_{g\in G}\cC_{g}$ such that $\cC_{g}\otimes \cC_{h}\subseteq \cC_{gh}$ and $\cC_{g}\ne 0$ for all $g\in G$. If $\cC$ is a fusion category, a \textit{G-graded extension of} $\cC$ is a (faithful) $G$-graded fusion category $\displaystyle \cD=\bigoplus_{g\in G} \cD_{g}$ with $\cD_{1}=\cC$.

\begin{thm}\cite[Theorem 7.7]{EtNiOs2010} (Faithful) G-graded extensions of a fixed fusion category $\cC$ are classified by monoidal 2-functors 
$$
    \underline{\underline{G}}\rightarrow \underline{\underline{\operatorname{BrPic}}}(\cC)\,.
$$
\end{thm}

Here $\underline{\underline{G}}$ is the monoidal 2-category whose objects are elements of $G$ and the 1 and 2 morphisms are all identities. The monoidal product is given by group multiplication on objects, and the obvious composition of identities. $\underline{\underline{\operatorname{BrPic}}}(\cC)$ is the monoidal 2-category whose objects are invertible bimodule categories, 1-morphisms are bimodule equivalences, and 2-morphisms are bimodule functor natural isomorphisms. The monoidal product is defined by taking the relative product of bimodules (functors, natural transformations) over $\cC$ \cite[Definition 4.5]{EtNiOs2010}.

This classification is fairly transparent. The table below gives a correspondences between the data of a monoidal 2-functor and the data of the extension (in the table below we neglect units).

\bigskip

\begin{center}
 \begin{tabular}{|| m{20em} |  m{17em}||}
 \hline
 \textbf{Data of monoidal 2-functor} & \textbf{Data of extension} $\cC\subseteq \cD$ \\ [0.5ex] 
 \hline\hline
 Assignment $g\mapsto \cD_{g}$ & Definition of $g$ components 
 \mbox{$\cD=\bigoplus_{g} \cD_{g}$} \\ 
 \hline
  Bimodule equivalences $T_{g,h}\colon\cD_{g}\boxtimes_{\cC}\cD_{h}\cong \cD_{gh}$ &  Definition of tensor product bi-functor $\otimes\colon \cD_{g}\boxtimes \cD_{h}\rightarrow \cD_{gh}$  \\
 \hline
 $a_{g,h,k}\colon T_{gh,k}\circ (T_{g,h}\boxtimes_{\cC} \operatorname{Id}_{k})\cong T_{g,hk}\circ (\operatorname{Id}_{g}\boxtimes_{\cC} T_{h,k})$ & Associator \mbox{$\alpha\colon (X_{g}\otimes Y_{h})\otimes Z_{k}\rightarrow X_{g}\otimes (Y_{h}\otimes Z_{k}) $}  \\  
 \hline
\end{tabular}
\end{center}

\bigskip

It is then shown that the coherence that the $a_{g,h,K}$ is required to satisfy is equivalent to the pentagon axiom for the corresponding associator.

While this result is conceptually straightforward, for it to be useful requires an understanding of the monoidal 2-category $\underline{\underline{\operatorname{BrPic}}}(\cC)$, which in general is a complicated beast. However, if we truncate the top level and take isomorphism classes of bimodule equivalence as 1-morphisms, we obtain the monoidal category $\underline{\operatorname{BrPic}}(\cC)$.

This monoidal category is easier to understand. For a fusion category $\cC$, let $\mathcal{Z}(\cC)$ denote the \textit{Drinfeld center} of $\mathcal{C}$ \cite{EtGeNiOs2015}. Given an invertible bimodule category $\mathcal{M}$, we have two equivalences $L_{\mathcal{M}},R_{\mathcal{M}}\colon \cZ(\cC)\cong \operatorname{End}_{\cC-\cC}(\mathcal{M})$ given by left and right multiplication respectively, where $\operatorname{End}_{\cC-\cC}(\mathcal{M})$ denotes the monoidal category of bimodule endofunctors of $\mathcal{M}$ \cite{EtGeNiOs2015}. The composition $L^{-1}_{\mathcal{M}}\circ R_{\mathcal{M}}$ gives a braided auto-equivalence $\alpha_{\mathcal{M}}\in \underline{\operatorname{Aut}}^{\mathrm{br}}_{\otimes}(\cZ(\cC))$.

\begin{thm}\cite[Theorem 1.1]{EtNiOs2010} The assignment $\mathcal{M}\mapsto \alpha_{\mathcal{M}}$ described above extends to a monoidal equivalence $\underline{\operatorname{BrPic}}(\cC)\cong \underline{\operatorname{Aut}}^{\mathrm{br}}_{\otimes}(\cZ(\cC))$.
\end{thm}
Thus given an extension with classifying functor $\underline{M}\colon \underline{\underline{G}}\rightarrow \underline{\underline{\operatorname{BrPic}}}(\cC)$, decategorifying canonically gives a monoidal functor
$$
    M\colon \underline{G}\rightarrow \underline{\operatorname{BrPic}}(\cC)\cong \underline{\operatorname{Aut}}^{\mathrm{br}}_{\otimes}(\cZ(\cC))\,.
$$

A monoidal functor $M\colon\underline{G}\to\underline{\operatorname{Aut}}^{\mathrm{br}}_{\otimes}(\cZ(\cC))$ is precisely a braided categorical action of $G$ on the Drinfeld center $\mathcal{Z}(\cC)$. The goal of this paper, is to recover the fusion rules of the extension from just the categorical action $M$. The reason this is useful is that often we use the above theorems in the reverse direction. 

Namely, suppose we want to \textit{construct} and extension from scratch. We can start from a categorical action $M\colon \underline{G}\rightarrow  \underline{\operatorname{Aut}}^{\mathrm{br}}_{\otimes}(\cZ(\cC))\cong \underline{\operatorname{BrPic}}(\cC)$. Then we can lift this to a monoidal 2-functor $\underline{M}\colon\underline{\underline{G}}\rightarrow \underline{\underline{\operatorname{BrPic}}}(\cC)$ if and only if a certain obstruction $o_{4}(M)\in H^{4}(G, \CC^{\times})$ vanishes. If it does, then we know an extension exists, and the possible associators form a torsor over $H^{3}(G, \CC^{\times})$.

In practice, one can often show the $o_{4}$ obstruction vanishes for general reasons (for example \cite{EtGal2018, GaJo2018}). In this situation, we know extensions exists, but it is often very difficult to say anything about the structure of such extensions in general. Thus new methods are required to work out the details of what an extension looks like when constructed in this way. The goal of this paper is precisely to provide such methods to determine the fusion rules of the extension.

In the sequel our method will require the existence of a spherical structure on the extension $\cD$. As mentioned in the introduction, to our knowledge there has been no general theory developed for constructing spherical structures on extensions, though it should certainly exist. In particular, there is a natural spherical analogue of the Braur-Picard groupoid, and one would expect spherical extensions would naturally be classified by monoidal 2-functors from $\underline{\underline{G}}$ to this category. However, as this theory has yet to be developed, thus having conditions on $\cC$ which  guarantee the existence of a spherical structures on our extensions automatically  makes it easier to apply our results.

We have the following proposition.

\begin{prop}\label{pseudounitaryext}
    Let $\cC$ be a pseudo-unitary fusion category, and let $\cC\subseteq \cD$ be a $G$-graded extension. Then $\cD$ is pseudo-unitary.
\end{prop}

\begin{proof}
Let $X\in \cD$ be in the $g$-graded component. Then choose a dual object $\bar{X}\in \cD_{g^{-1}}$, and solutions to the duality equations $R_{X}, \bar{R}_{X}, R^{*}_{X}, \bar{R}^{*}_{X}$ (here we use the notation preceding equation \ref{paireddimension}). Then $X\otimes \bar{X}\in \cC$ is canonically equipped with the structure of a connected special Frobenius algebra, with multiplication $m:=1_{X}\otimes R^{*}_{X} \otimes 1_{\bar{X}}$, co-multiplication $\Delta:=1_{X}\otimes R_{X} \otimes  1_{\bar{X}}$, unit $\iota:=\bar{R}_{X}$ and counit $\epsilon:=\bar{R}^{*}_{X}$. 

Then this algebra is special, with constants $m\circ \Delta=(R^{*}_{X}\circ R_{X}) 1_{X\otimes \bar{X}}$ and $\epsilon\circ \iota=\bar{R}^{*}_{X}\circ \bar{R}_{X}$. 
Thus the invariant quantity $\beta$ associated to any special Frobenius algebra defined by $\epsilon\circ m\cdot \Delta\circ \iota=\beta 1_{\tu}$ in this case is precisely the paired dimension $d_{\{X,\bar{X}\}}$. 

Since $X$ is simple, the algebra $X\otimes \bar{X}$ is connected (also called \textit{haploid} in the literature). Thus by \cite[Corollary 3.10]{FuRuSc2002}, this algebra will be symmetric with respect to \textit{any} spherical structure for which the spherical dimension of $X\otimes \bar{X}$ is non-zero. But symmetric special Frobenius algebras $A$ satisfy $\beta=d_{A}$. In particular, choosing the canonical pseudo-unitary spherical structure, the above shows that the connected Frobenius algebra $X\otimes \bar{X}$ is symmetric, hence $d^{\cD}_{\{X,\bar{X}\}}=\beta=d^{\cC}_{X\otimes \bar{X}}=\text{FPdim}_{\cC}(X\otimes \bar{X}) =\text{FPdim}_{\cD}(X\otimes \bar{X})=\text{FPdim}_{\cD}(X)^{2}$, where we have used pseudo-unitarity for the third equality. Thus $\cD$ is pseudo-unitary.
\end{proof}

\section{Recovering fusion rules from the Lagrangian algebra}\label{recovfusionrules}

In this section, we will explain how the fusion rules of a fusion category can be derived from a pair of algebraic operations on the vector space $\End_{\cZ(\cC)}(I(\mathbbm{1}))$. We use these results together with the facts we've assembled about equivariantizations to describe the fusion rules for extensions. Our conventions for half-braidings and spherical structures follow \cite{Mg2003}, \cite{Mg2003II}. We make extensive use the graphical calculus for spherical fusion categories. We use the ``optimistic" convention, so that our diagrams are read bottom to top.

We refer the reader to \cite{FuRuSc2002} for definitions concerning algebras in tensor categories and their various adjectives. We warn the reader that following \cite{EtGeNiOs2015} we use the word \textit{connected} to mean $\dimCC\cC(\tu, A)=1$, whereas in many references (including \cite{FuRuSc2002}) the word \textit{haploid} is used.  Let $A$ be any commutative, connected special Frobenius algebra in a braided spherical fusion category with $d_{A}\ne 0$, normalized so that 
$$
 \tikzmath{
      \draw[very thick]
      (-1,0) arc (180:360:1)-- (1,1) arc (0:180:1)--(-1,0)
      (0,-2) node [below] {$\scriptstyle A$} --(0,-1)
      (0,3) node [above] {$\scriptstyle A$}--(0,2);
      \mydotw{(0,2)}
      \mydotw{(0,-1)}
    }=
    \tikzmath{
      \draw[very thick]
      (0,-2) node [below] {$\scriptstyle A$}--
      (0,3) node [above] {$\scriptstyle A$};
    }
    \, \quad,\ 
   \varepsilon \cdot i=
    \tikzmath{
      \draw[very thick]
        (0,0)-- node [right] {$\scriptstyle A$} (0,1);
      \mydotw{(0,0)}
      \mydotw{(0,1)}
    }
    =d_{A}\,.
$$
We define the \textit{convolution product} on $\End_\cC(A)$ by 
\begin{equation}\label{defconvolutionproduct}
\tikzmath{
      \draw[very thick]
      (0,-2) node [below] {$\scriptstyle A$}--
      (0,3) node [above] {$\scriptstyle A$};
      \node[coupon] at (0,0.5) {$\scriptstyle a\ast b$};
    }
    :=
    \tikzmath{
      \draw[very thick]
      (-1,0) arc (180:360:1)-- (1,1) arc (0:180:1)--(-1,0)
      (0,-2) node [below] {$\scriptstyle A$} --(0,-1)
      (0,3) node [above] {$\scriptstyle A$}--(0,2);
      \node[coupon] at (-1,0.5) {$\scriptstyle a$};
      \node[coupon] at (1,0.5) {$\scriptstyle b$};
      \mydotw{(0,2)}
      \mydotw{(0,-1)}
    }\,.
\end{equation}
This operation on $\End_\cC(A)$ makes it into an associative, commutative algebra. The \emph{unit} with respect to the convolution product is given by 
$$
    i\circ \varepsilon = 
    \tikzmath{
        \draw[very thick] 
        (0,3) node [above] {$\scriptstyle A$}
        --(0,2)
        (0,1)--(0,0) node [below]
        {$\scriptstyle A$};
        \mydotw{(0,1)}
        \mydotw{(0,2)}
    }
$$
We note that $\End_\cC(A)$ also has the usual composition, and thus we have two operations on this vector space $(\End_\cC(A), \circ, \ast)$. 
By \cite[Corollary 2.5]{BiDa2018}, $(\End_\cC(A), \ast)$ is a semi-simple commutative algebra and is thus isomorphic to $\CC^{n}$. Thus we can ``diagonalize'' the multiplication by finding minimial idempotents. We note this idempotents give a canonical basis for the vector space $\End_\cC(A)$ and $(\End_\cC(A),\circ)$ becomes a based algebra.

\begin{lem}\label{isomorphism} 
    Let $A, B$ be connected special Frobenius algebras with non-zero dimension, normalized as above, which are isomorphic as algebras. Then 
    
    $$(\End_\cC(A),\circ, \ast)\cong (\End_\cC(B), \circ, \ast).$$
\end{lem}

\begin{proof}
By \cite[Corollary 3.10]{FuRuSc2002}, $A$ and $B$ are symmetric, hence by \cite[Theorem 3.6]{FuRuSc2002} there is a unique comultiplication with the desired  normalization. Therefore any algebra intertwiner $\psi\in \cC(A, B)$ must  also intertwine the comultiplications. Indeed, if $m_{B}\in \cC(B\otimes B, B)$ and $n_{B}\in \cC(B, B\otimes B)$ denote the normalized Frobenius multiplication and comultiplication for $B$ respectively, then $(\psi^{-1}\otimes \psi^{-1})\circ n_{B}\circ \psi\in \cC(A, A\otimes A)$ provides an appropriately normalized comultiplication for $m_{A}$ and therefore must be $n_{A}$ (a similar argument applies to counits). Thus the map $\End_\cC(A)\rightarrow \End_\cC(B)$, $f\mapsto \psi \circ f \circ \psi^{-1}\in \End_\cC(B)$ is an isomorphism with respect to $\circ$ and $\ast$.
\end{proof}

\subsection{The canonical Lagrangian algebra}

Recall that on object in the Drinfeld center $Z(\cC)$ consists of pairs $(Y, \phi_{Y})$ where $Y\in \cC$ and $\phi_{Y}$ is a natural isomorphisms from the functor $Y\otimes \cdot\rightarrow \cdot \otimes Y$ called \textit{half-braidings} satisfying a version of the hexagon coherence \cite{Mg2003II}. Morphisms between such pairs consist of morphisms between the underlying objects which intertwine the half-braidings. The functor from $Z(\cC)$ to $\cC$ which sends a pair $(Y,\phi_{Y})$ to the object $Y$ and morphisms to themselves is called the \textit{forgetful functor}, denoted $\operatorname{F}\colon Z(\cC)\rightarrow \cC$. 

Let $\cC$ be a spherical fusion.
Let us once and for all pick a square root $\sqrt{d_{X}}$ for each $X\in \Irr(\cC)$. The forgetful functor admits an adjoint  $I\colon\cC\rightarrow Z(\cC)$ (see Remark \ref{adjoint remark}). By \cite{KiBa2010}, we can represent $I$ with the following explicit formula:
\begin{align}\label{definehalfbraiding}
I(X)&:=(\textstyle\bigoplus_{Y\in \text{Irr}(\cC)} Y\otimes X \otimes \overline{Y},\ \phi_{I(X)})
\,,
\\
\phi_{I(X),W}&:= 
   \bigoplus_{Y,Z\in\Irr(\cC)}
    \sum_i \sqrt{d_Y}\sqrt{d_Z}
    \tikzmath{
        \draw[thick]
        (0,0) node [below] {$\scriptstyle  Y$ }--(0,4)node [above] {$\scriptstyle Z$ }
        (1,0) node [below] {$\scriptstyle  X$ }--(1,4)node [above] {$\scriptstyle X$ }
        (0,2) to [in=270 ,out=110] (-1,4) node [above] {$\scriptstyle W$}
        (2,0)node [below] {$\scriptstyle \bar Y$ }--(2,4) node [above] {$\scriptstyle \bar Z$ }
        (2,2) to [out=290, in=90] (3,0) node [below] {$\scriptstyle W$}
        ;
        \mydotw{(0,2)}
        \mydotw{(2,2)}
        \node at (0,2) [right] {$\scriptstyle i$};
        \node at (2,2) [right] {$\scriptstyle i^{\bullet}$};
    }\,,
\end{align}
here $\{i\}$ is a basis for $\cC(Y, W\otimes Z)$ and $\{i^{\bullet}\}\subseteq \cC(\bar{Y}\otimes W, \bar{Z})$ is a dual basis defined via
\begin{equation}\label{graphicalpairing}
\delta_{i,j}:=\tikzmath{
        \draw[thick]
        (2,1) arc (360:180:1)
        (0,1) -- node [left]
        {$\scriptstyle \bar Y$}(0,3)
        arc (180:0:1)
        (0,3) to [in=110,out=290]
        node [below] {$\scriptstyle W$}
        (2,1)
        (2,1)--node [right]
        {$\scriptstyle Z$}(2,3)
        ;
        \mydotw{(2,1)}
        \mydotw{(0,3)}
        \node at (2,1) [right] {$\scriptstyle i$};
        \node at (0,3) [left] {$\scriptstyle j^{\bullet}$};
        }\,.
\end{equation}
        
\begin{rem}\label{basisdualbasis}
Here and throughout this paper, we make extensive use of the following elementary fact: for a finite dimensional vector space $V$, there is a canonical element in $V\otimes V^{*}$ defined by $\sum_{i} b_{i}\otimes b^{*}_{i} $ where $\{b_{i}\}\subseteq V$ is a basis and $\{b^{*}_{i}\}\subseteq V^{*}$ is the dual basis satisfying $b^{*}_{i}(b_{j})=\delta_{i,j}$. Though we have picked a basis to define it, this element does not actually depend on the \textit{choice} of basis. Throughout this text we will have vector spaces of morphisms arising from our category that are linked together with (various) non-degenerate pairings. Since our monoidal categories are linear, inserting morphisms into planar diagrams is multi-linear. It follows that any of our pictures which have a summation over diagrams which consist of a basis element together with its dual with respect to some non-degenerate pairing, the resulting overall morphism will not depend on the choice of this basis. This will be a key ingredient of many of our arguments.
\end{rem}

In \cite[Theorem 2.3]{KiBa2010}, the authors establish $I$ as a adjoint to the forgetful functor. In particular, they provide a (left) adjunction between $I$ and $F$ by establishing a bijection $\cC(X, F(Y, \phi_{Y}) )\cong Z(\cC)(I(X), (Y, \phi_{Y}) )$  defined by 
\begin{equation}\label{adjunctioncenter}
\tikzmath{
    \draw[ultra thick]
        (0,2)--(0,4);
    \draw[thick]
        (0,0) node [below] {$\scriptstyle X$} -- 
        (0,4) node [above] {$\scriptstyle Y$};
    \node[coupon] at (0,2) {$\scriptstyle f$};
}
\longmapsto
\bigoplus_{Z\in\Irr(\cC)}
\sqrt{d_Z}
\tikzmath{
    \draw[ultra thick]
        (1,2)--(1,4);
    \draw[thick] 
        (1,0) node [below] {$\scriptstyle X$} -- 
        (1,4) node [above] {$\scriptstyle Y$};
    \node[coupon] at (1,2) {$\scriptstyle f$};
    \node at (1,3) [above right] 
        {$\scriptstyle \phi_{Y,\bar Z}$};
    \draw[ultra thick, white, double] 
        (0,2.25) arc(180:0:1);
    \draw[thick] 
        (0,0) node [below] {$\scriptstyle Z$}
        --(0,2.25)
        arc(180:0:1)--
        (2,0) node [below] {$\scriptstyle \bar Z$} ;
}\,.
\end{equation}
The object $I(\tu)$ is endowed with the structure of a (symmetric) special Frobenius algebra in $Z(\cC)$, with structure maps
\begin{align}\label{Lagrangianalgmult}
    \tikzmath{
        \draw[very thick]
            (0,.5) node [below] {$\scriptstyle I(\mathbbm{1})$ } --(0,1) arc (180:0:1) --(2,.5) node [below] {$\scriptstyle I(\mathbbm{1})$ } 
            (1,2)--(1,4) node [above]
            {$\scriptstyle I(\mathbbm{1})$ }
            ;            
            \mydotw{(1,2)}
    }
    &=
    \bigoplus_{X\in\Irr(\cC)}
    \frac{1}{\sqrt{d_{X}}}
    \tikzmath{
        \draw[thick] 
            (0,0.5) node [below] {$\scriptstyle \bar X$ } -- (0,1) arc (180:0:1) --(2,.5)
            node [below] {$\scriptstyle X$ }
            (-1,.5) node [below] {$\scriptstyle  X$} -- (-1,1) arc (180:135:2)
            to [in=-90,out=45]
            (.5,3.5) --(.5,4) node [above] {$\scriptstyle X$}
              (3,.5) node [below] {$\scriptstyle  \bar X$}  -- (3,1) arc (0:45:2) to [in=-90,out=135] (1.5,3.5)--(1.5,4) node [above] {$\scriptstyle \bar  X$}
            ;}
    \,,
    \\
     \tikzmath{
    \draw[very thick]
    (0,0) to (0,2) node [above] {$\scriptstyle I(\mathbbm{1})$};
    \mydotw{(0,0)};
    }
    &=
    \bigoplus_{X\in \Irr(\cC)} \sqrt{d_{X}}\tikzmath{
    \draw[thick]
    (2,1)node [above] {$\scriptstyle \bar{X}$} --(2,0) 
    arc(360:180:1) 
    (0,0)--(0,1) node [above] {$\scriptstyle X$};
    }
\end{align}
The comultiplication and counit are given by the reflected diagrams of the multiplication and unit maps respectively, with the same normalizing coefficients. Thus $I(\tu)$ is a connected special Frobenius algebra normalized as in the previous section (note $d_{I(\tu)}=\sum_{X\in \Irr(\cC)} d^{2}_{X}=\dim(\cC)>0$).



From above we see
$$
    \End_{\cZ(\cC)}(I(\mathbbm{1}))
    \cong
    \cC(\mathbbm{1}, F\circ I(\mathbbm{1}))\cong \bigoplus_{X\in \Irr(\cC)} \cC(\mathbbm{1}, X\otimes \bar{X})\,.
$$
We have a basis for $\bigoplus_{X\in \Irr(\cC)} \cC(\mathbbm{1}, X\otimes \bar{X}) $ consisting of cups. Namely, set 

$$
    r_{Y}:=\sqrt{d_{Y}}\tikzmath{
    \draw[thick]
    (2,1)node [above] {$\scriptstyle Y$} --(2,0) 
    arc(360:180:1) 
    (0,0)--(0,1) node [above] {$\scriptstyle \bar Y$};
    }\,.
$$
Then $\{r_{Y}\}$ form a basis for $\bigoplus_{X\in \text{Irr}(\cC)} \cC(\mathbbm{1}, X\otimes \bar{X}) $ (note that here we are implicitly using the spherical structure to identify $\bar{Y}$ with its representative in $\Irr(\cC)$ and $\bar{\bar{Y}}$ with $Y$). Now we consider the image of $r_{Y}$ under the canonical adjunction from Equation \ref{adjunctioncenter}, which from Equation \ref{definehalfbraiding} reads
\begin{equation}\label{eYdefinition}
    e_Y=
    \bigoplus_{X, Z \in\Irr(\cC)}
    \sum_j d_Y\sqrt{d_X}\sqrt{d_Z} 
    \tikzmath{
        \draw[thick]
        (0,2) to [in=270 ,out=110] (-1,3) arc (0:180:1) -- (-3,0) node [below] {$\scriptstyle  X$}
        (0,4)node [above] {$\scriptstyle Z$ }--(0,2) arc (180:360:1) --node [left] { }
        (2,4) node [above] {$\scriptstyle \bar Z$ }
        (2,2) to [out=290, in=90] (3,0) node [below] {$\scriptstyle \bar X$}
        (-0.3,0.5) node [above] {$\scriptstyle \bar Y$ };
        \mydotw{(0,2)}
        \mydotw{(2,2)}
        \node at (0,2) [right] {$\scriptstyle j$};
        \node at (2,2) [right] {$\scriptstyle j^{*}$};
    }
    =
    \bigoplus_{X, Z \in\Irr(\cC)}
    \sum_i d_{Y}\sqrt{d_{X}}\sqrt{d_{Z}}
    \tikzmath{
        \draw[thick]
        (0,0) node [below] {$\scriptstyle  X$ }--(0,4)node [above] {$\scriptstyle Z$ }
        (0,1) to [in=260,out=80] 
        node [above] {$\scriptstyle Y$}
        (2,3)
        (2,0)node [below] {$\scriptstyle \bar X$ }--(2,4) node [above] {$\scriptstyle \bar Z$ }
        ;
        \mydotw{(0,1)}
        \mydotw{(2,3)}
         \node at (0,1) [right] {$\scriptstyle i^{\bullet}$};
        \node at (2,3) [right] {$\scriptstyle i$};
    }.\,
\end{equation}
The summation of $i$ is over a basis for $\cC(Y\otimes \bar{X},\bar{Z})$, and $i^{*}$ is a dual basis with respect to a reflected version of the graphical pairing defined \ref{graphicalpairing}, given explicitly by

\begin{equation}
\delta_{i,j}:=\tikzmath{
        \draw[thick]
        (2,1) arc (360:180:1)
        (0,1) -- node [left]
        {$\scriptstyle \bar X$}(0,3)
        arc (180:0:1)
        (0,1) to [in=240,out=70]
        node [below] {$\scriptstyle Z$}
        (2,3)
        (2,1)--node [right]
        {$\scriptstyle Y$}(2,3)
        ;
        \mydotw{(2,3)}
        \mydotw{(0,1)}
        \node at (2,3) [right] {$\scriptstyle i^{*}$};
        \node at (0,1) [left] {$\scriptstyle j$};
        }\,.
\end{equation}

To see that the graphical terms in equation \ref{eYdefinition} are equal, note that the rotated $j$ and $j^{\bullet}$ form a basis and dual basis with respect to the graphical pairing as well (this can be easily seen using sphericality with the $Z$ string and using pivotality to apply a $2\pi$ rotation to the resulting twisted morphism $j$, which results in the graphical pairing from \ref{graphicalpairing}). Thus we can apply Remark \ref{basisdualbasis}.

A straightforward computation then gives us the following:
\begin{enumerate}
    \item 
    $\{e_{Y}\}$ forms a basis for $\End_{Z(\cC)}(I(\tu))$,
    \item
    $e_{Y}\ast e_{Z}=\delta_{Y,Z} e_{Y}$.
\end{enumerate}

\medskip 

To see the second point,
$$
    e_Y\ast e_Z = 
    \bigoplus_{X, V \in\Irr(\cC)}
    \sum_{i,j} d_{Y}d_Z 
    \sqrt{d_{X}}\sqrt{d_{V}}
    \tikzmath{
        \draw[thick]
        (0,0) node [below] {$\scriptstyle  X$ }--(0,4)node [above] {$\scriptstyle V$ }
        (0,1) to [in=260,out=80] 
        node [above] {$\scriptstyle Y$}  (2,3)
        (2,1) node [below left] {$\scriptstyle \bar X$ }--(2,3)
        node [above left] {$\scriptstyle \bar V$ }
        arc (180:0:1) -- (4,1) arc (360:180:1)
        (4,1) to [in=260,out=80] 
        node [above] {$\scriptstyle Z$}  (6,3)
        (6,0) node [below] {$\scriptstyle  \bar X$ }--(6,4)node [above] {$\scriptstyle \bar V$ }
        ;
        \mydotw{(0,1)}
        \mydotw{(2,3)}
        \mydotw{(4,1)}
        \mydotw{(6,3)}
        \node at (0,1) [right] {$\scriptstyle i^{\bullet}$};
        \node at (2,3) [right] {$\scriptstyle i$};
        \node at (4,1) [right] {$\scriptstyle j^{\bullet}$};
        \node at (6,3) [right] {$\scriptstyle j$};
    }\,
    = \delta_{Y,Z} \cdot
    \bigoplus_{X, V \in\Irr(\cC)}
    \sum_i d_{Y}\sqrt{d_{X}}\sqrt{d_{V}}
    \tikzmath{
        \draw[thick]
        (0,0) node [below] {$\scriptstyle  X$ }--(0,4)node [above] {$\scriptstyle V$ }
        (0,1) to [in=260,out=80] 
        node [above] {$\scriptstyle Y$}
        (2,3)
        (2,0)node [below] {$\scriptstyle \bar X$ }--(2,4) node [above] {$\scriptstyle \bar V$ }
        ;
        \mydotw{(0,1)}
        \mydotw{(2,3)}
         \node at (0,1) [right] {$\scriptstyle i^{\bullet}$};
        \node at (2,3) [right] {$\scriptstyle i$};
    }\,
$$
since
$$
   \tikzmath{
        \draw[thick]
        (0,-1) node [below] {$\scriptstyle Y$}--
        (0,1) to [in=260,out=80]     (2,3)
        (2,1) node [below left] {$\scriptstyle \bar X$ }--(2,3)
        node [above left] {$\scriptstyle \bar V$ }
        arc (180:0:1) -- (4,1) arc (360:180:1)
        (4,1) to [in=260,out=80] 
        (6,3)--(6,5)   node [above] {$\scriptstyle Z$} 
        ;
        \mydotw{(2,3)}
        \mydotw{(4,1)}
        \node at (2,3) [right] {$\scriptstyle i$};
        \node at (4,1) [right] {$\scriptstyle j^{\bullet}$};
    }\,
    =\delta_{Y,Z} \frac1{d_X}\  
   \tikzmath{
        \draw[thick]
        (-4,-1) node [below] {$\scriptstyle Y$}--(-4,5) node [above] {$\scriptstyle Y$}
        (0,1) to [in=260,out=80] 
        node [above] {$\scriptstyle Y$}  (2,3)
        (2,1) node [below left] {$\scriptstyle \bar X$}--(2,3)
        node [above left] {$\scriptstyle \bar V$}
        arc (180:0:1) -- (4,1) arc (360:180:1) (4,1) 
        to [in=260,out=80] 
        node [above] {$\scriptstyle Y$}  (6,3) 
        arc (0:180:4) -- (-2,0) arc (180:360:1)--(0,1)
        ;
        \mydotw{(2,3)}
        \mydotw{(4,1)}
        \node at (2,3) [right] {$\scriptstyle i$};
        \node at (4,1) [right] {$\scriptstyle j^{\bullet}$};
    }\,
    =\delta_{Y,Z} \frac1{d_Y}\  
    \tikzmath{
        \begin{scope}[yscale=-1]
        \draw (-2,-1) node [above] {$\scriptstyle Y$}--(-2,5) node [below] {$\scriptstyle Y$};
        \draw[thick]
        (2,1) arc (360:180:1)
        (0,1) -- node [left]
        {$\scriptstyle V$}(0,3)
        arc (180:0:1)
        (0,3) to [in=110,out=290]
        node [below] {$\scriptstyle Y$}
        (2,1)
        (2,1)--node [right]
        {$\scriptstyle Z$}(2,3)
        ;
        \mydotw{(2,1)}
        \mydotw{(0,3)}
        \node at (2,1) [right] {$\scriptstyle i$};
        \node at (0,3) [left] {$\scriptstyle j^{\bullet}$};
        \end{scope}
        }\,.
    =\delta_{Y,Z} \frac1{d_Y}\delta_{i,j}\ 
       \tikzmath{
        \draw[thick]
        (-4,-1) node [below] {$\scriptstyle Y$}--(-4,5) node [above] {$\scriptstyle Y$};
        }
    \,.
$$

In other words, the collection $\{e_{Y}\}$ diagonalizes the convolution product. 

\begin{prop}\label{fusioncomposition}
    $e_{Y}\circ e_{Z}=\sum_{X\in \Irr(\cC)} \frac{d_{Y}d_{Z}}{d_{X}} N^{X}_{YZ} e_{X}$. 
\end{prop}
\begin{proof}
We see that 
$$
    e_{Y}\circ e_{Z}=\bigoplus_{P,R\in \Irr(\cC)} \sum_{Q\in \Irr(\cC)} \sum_{i,j} d_{Y} d_{Z} d_{Q} \sqrt{d_{R}} \sqrt{d_{P}} 
\tikzmath{
        \draw[thick]
        (0,0) node [below] {$\scriptstyle  P$ }--(0,5)node [above] {$\scriptstyle R$ }
        (0,1) to [in=260,out=80] 
        node [above] {$\scriptstyle Z$}
        (2,2)
        (0,3) to [in=260,out=80] 
        node [above] {$\scriptstyle Y$}
        (2,4)
        (2,0)node [below] {$\scriptstyle \bar P$ }--(2,5) node [above] {$\scriptstyle \bar R$ }
        ;
        \mydotw{(0,3)}
        \mydotw{(2,4)}
        \mydotw{(0,1)}
        \mydotw{(2,2)}
        \node at (0,1) [right] {$\scriptstyle i^{\bullet}$};
        \node at (2,2) [right] {$\scriptstyle i$};
        \node at (0,3) [right] {$\scriptstyle j^{\bullet}$};
        \node at (2,4) [right] {$\scriptstyle j$};
        \node at (0,2) [left] {$\scriptstyle Q$};
        \node at (2,3) [right] {$\scriptstyle \bar Q$};
    }\,.
$$
However, the sets 
$$d_{Q} \tikzmath{
        \draw[thick]
        (-1,0) node [below] {$\scriptstyle Y$} --(1,2)
        (1,2) to  (1,3) node [above] {$\scriptstyle \bar{R}$}
        (2,1) to  (1,2)
        (1,0)node [below] {$\scriptstyle Z$}--(2,1)
        (3,0) node [below] {$\scriptstyle \bar{P}$} --(2,1);
        \mydotw{(1,2)};
        \mydotw{(2,1)};
        \node at (1.3,1.7) [right] {$\scriptstyle \bar Q$};
        \node at (1,2) [left] {$\scriptstyle j$};
        \node at (2,1) [left] {$\scriptstyle i$};
        },\quad  
d_{U} \tikzmath{
        \draw[thick]
        (-1,0) node [below] {$\scriptstyle Y$} --(1,2)
        (1,2) to  (1,3) node [above] {$\scriptstyle \bar{R}$}
        (1,0)node [below] {$\scriptstyle Z$}--(0,1)
        (3,0) node [below] {$\scriptstyle \bar{P}$} --(1,2);
        \mydotw{(1,2)};
        \mydotw{(0,1)};
        \node at (0.6,1.8) [left] {$\scriptstyle U$};
        \node at (1,2) [right] {$\scriptstyle k$};
        \node at (0,1) [left] {$\scriptstyle l$};
        }
$$
\noindent as $Q, U$ runs over $\Irr(\cC)$ and $i,j, k,l$ run over the graphically normalized bases both form a (graphically normalized) basis for $\cC(Y\otimes Z\otimes \bar{P}, \bar{R})$. Since the first basis set appears together with its dual in the above expression, we can replace it with the latter (Remark \ref{basisdualbasis}), to obtain
\begin{equation*}
    e_{Y}\circ e_{Z}=\bigoplus_{P, R \in\Irr(\cC)}
    \sum_i d_{Y}d_{Z}N^{U}_{Y,Z}\sqrt{d_{R}}\sqrt{d_{P}}
    \tikzmath{
        \draw[thick]
        (0,0) node [below] {$\scriptstyle  P$ }--(0,4)node [above] {$\scriptstyle R$ }
        (0,1) to [in=260,out=80] 
        node [above] {$\scriptstyle U$}
        (2,3)
        (2,0)node [below] {$\scriptstyle \bar P$ }--(2,4) node [above] {$\scriptstyle \bar R$ }
        ;
        \mydotw{(0,1)}
        \mydotw{(2,3)}
         \node at (0,1) [right] {$\scriptstyle k^{\bullet}$};
        \node at (2,3) [right] {$\scriptstyle k$};
    }
    = \sum_{U\in \Irr(\cC)}\frac{d_{Y}d_{Z}}{d_{U}}N^{u}_{Y,Z} e_{U}\,.\qedhere
 \end{equation*}    
\end{proof}
Now let $\cC$ be a fusion category and $d$ a spherical dimension function. Then consider $(K_{0}(\cC), \cdot, \ast_{d})$, where $K_{0}(\cC)=\mathbbm{C}[\Irr(\cC)]$, $[X]\cdot[Y]=\sum_{Z\in \Irr(\cC)} N^{Z}_{XY}[Z]$, and $[X]\ast_{d}[Y]=\frac{\delta_{X,Y}}{d_{X}} [X]$. Then the above proposition and a straightforward computation gives us the following corollary:
\begin{cor}
    \label{recovfusionrulesthm} The assignment $e_{X}\mapsto d_{X}[X]$ gives an isomorphism $$\textbf{}(\operatorname{End}_{\cZ(\cC)}(I(\mathbbm{1})), \circ, \ast)\cong (K_{0}(\cC), \cdot, \ast_{d})\,.$$
\end{cor}

Thus if we have the algebraic structure $(\End_{\cZ(\cC)}(I(\tu)), \circ, \ast)$ \textit{and we know the spherical dimension function}, we can determine the fusion rules by rescaling the canonical basis. Unfortunately this is not information we will have a-priori.

In the extension construction described in Section \ref{extension theory} the input is a categorical action $M\colon\underline{G}\rightarrow \underline{\operatorname{Aut}}^\mathrm{\mathrm{br}}_{\otimes}(\cZ(\cC))$. Suppose $o_{4}(M)$ vanishes, so there exists a (several) extension $\cC\subseteq \cD$. We would like to compute the fusion rules for this extension. We will assume the extension $\cD$ admits a spherical structure. In the next section we will show how to compute $(\End_{\cZ(\cD)}(I(\mathbbm{1})), \circ, \ast)$.

As we've mentioned, a-priori this is not quite enough to reconstruct the fusion rules, since we don't know  \textit{which} dimension function $d$ our basis is scaled with respect to! Indeed, we do not even know the fusion rules of $\cD$ yet, so trying to determine the possible dimension functions is premature.

However, we \textit{can} use Proposition \ref{fusioncomposition} to determine the \textit{square} of the dimensions (i.e. the paired dimensions) as follows: first determine the canonical basis element acting as the unit under composition, $e_\tu$ which is straightforward. For each $e_{Y}$, there will be a unique element $e_{\bar{Y}}$ such that $e_{Y}\circ e_{\bar{Y}}$ has a coefficient of $e_{\mathbbm{1}}$. This coefficient will be $d^{2}_{Y}=d_{\{Y,\bar{Y}\}}>0$, the canonical paired categorical dimension (see equation \ref{paireddimension}). Recall the positive square root (i.e. the fusion dimension) is denoted $d^{+}_{Y}$. We have $d_{Y}=\gamma_{Y} d^{+}_{Y}$, where $\gamma_{\cdot}\colon \Irr(\cD)\rightarrow \{\pm 1\} $ is determined by (and determines) the spherical structure. More explicitly, to have a spherical structure we need to find a function $\gamma_{\cdot}$ as above such that whenever $Z\prec X\otimes Y$, we have $\frac{\gamma_{X}\gamma_{Y}}{\gamma_{Z}}=T^{Z}_{XY}$ is the pivotal operator (See \cite[Theorem 5.4]{Ba16}). In particular, for such a spherical structure defined by $\gamma_{\cdot}$ to exist, we must have $T^{Z}_{XY}=\pm 1$ is constant whenever $Z\prec X\otimes Y$.

In any case, when we have a spherical structure, the spherical dimensions necessarily satisfy
$$
    \frac{d_{X}d_{Y}}{d_{Z}}=\pm \frac{d^{+}_{X}d^{+}_{Y}}{d^{+}_{Z}}\,.
$$

With the associative algebra $\textbf{}(\operatorname{End}_{\cZ(\cC)}(I(\mathbbm{1})), \circ)\cong (K_{0}(\cC), \cdot)\,.$ and the scaled basis elements described above in hand, we can compute the coefficient of $e_{Z}$ in $e_{X}\circ e_{Y}$, which is $C^{Z}_{XY}=\frac{d_{X}d_{Y}}{d_{Z}}N^{Z}_{XY}$.  We can also determine $\frac{d^{+}_{X}d^{+}_{Y}}{d^{+}_{Z}}$ as described above, and thus the fusion rule can be recovered as
\begin{equation}\label{fusioncoeff}
N^{Z}_{XY}=\left| C^{Z}_{X,Y} \frac{d^{+}_{Z}}{d^{+}_{X}d^{+}_{Y}}\right|\,.
\end{equation}

\medskip

\paragraph{\textbf{Summary of preceding discussion.}}\textit{Suppose we are given a vector space $V$ and two bilinear operations $\circ, \ast$ such that the triple $(V, \circ, \ast)$ is isomorphic to $(K_{0}(\cC), \cdot, \ast_{d})$ for some spherical dimension function $d$. Using the above procedure we can recover the fusion rules without determining $d$. While our method requires the existence of a spherical structure to produce the fusion rules, it does not require the choice of a specific one}

\medskip

\begin{rem}\label{Nospherical}
Above we assumed the existence of a spherical structure to derive our result. In the hypothetical case that there exists a fusion category which admits no spherical structure, the convolution product and composition product still make sense for the Frobenius algebra $I(\tu)$. One can show, however, that instead of recovering the fusion rules using our procedure above, we recover the \textit{signed} fusion rules $\operatorname{Tr}(T^{Z}_{XY})$, where again $T^{Z}_{XY}:\cC(Z,X\otimes Y)\rightarrow \cC(Z, X\otimes Y)$ is the pivotal operator \cite{Ba16}. The operator $T^{Z}_{XY}$ has order $2$ hence its eigenvalues are $\pm 1$. If we let $N^{Z,+}_{XY}$ and $N^{Z,-}_{XY}$ be the dimensions of the $1$ and $-1$ eigenspaces respectively, then $Tr(T^{Z}_{XY})=N^{Z,+}_{XY}-N^{Z,-}_{XY}$ which motivates our terminology. It does not seem to be possible to recover the fusion rules from this information unless we have $T^{Z}_{XY}=\pm 1$ for every triple of simple objects.
\end{rem}

\subsection{\texorpdfstring{$G$}{G}-extensions}\label{Gexts}
In the previous section, we showed how to recover the fusion rules of a fusion category from the algebraic structure of the Lagrangian algebra $I(\mathbbm{1})\in \cZ(\cC)$. Given a $G$-extension $\cC\subseteq \cD$, and have a categorical action $M\colon\underline{G}\rightarrow \underline{\operatorname{Aut}}^\br_{\otimes}(\cZ(\cC))$. The point of this section is to show how to describe the endomorphisms, convolution, and composition product of the canonical Lagrangian algebra for $\cD$ in terms of the data of the Lagrangian algebra for $\cC$ and the categorical action $M$. Our approach is based on the results of \cite{GeNaNi2009}, which realize the Drinfeld center $\cZ(\cD)$ as a certain equivariantization.

Recall from Section \ref{extension theory} that given a $G$-extension $\cC\subseteq \cD$, we have a canonically associated categorical action $M\colon \underline{G}\rightarrow \underline{\operatorname{Aut}}^\br_{\otimes}(Z(\mathcal{C}))$. From \cite[Theorem 3.3]{GeNaNi2009}, we have that the relative center of the extension $Z_{\mathcal{C}}(\cD)$ is a $G$-crossed braided extension of $Z(\mathcal{C})$, whose $G$-action on the trivial component is precisely the canonical action $M$.

Furthermore, we have $Z(\cD)=Z_{\mathcal{C}}(\cD)^{G}$ \cite[Theorem 3.5]{GeNaNi2009}. The forgetful functor $ Z(\cD)\rightarrow \cD$ factorizes as the composition of the forgetful functor $Z(\cD)\rightarrow Z_{\cC}(\cD)$ with $Z_{\cC}(\cD)\rightarrow \cD$, and thus its adjoint factors as a composite of the respective adjoints. However, upon identification of $Z(\cD)$ with $Z_{\mathcal{C}}(\cD)^{G}$, the first forgetful functor (which forgets the half-braiding with all $\cD$ and just remembers the half-braiding with the trivial component) is identified with the equivariant forgetful functor, which simply forgets the equivariant structures on objects.

Let $I_{\cD}\colon\cD\rightarrow Z(\cD)$ and $I_{\cC}\colon \cC\rightarrow Z(\cC)$ the denote the adjoints of the forgetful functors and $I_{G}: Z_{\cC}(\cD)\rightarrow Z_{\cC}(\cD)^{G}\cong Z(\cD)$ as defined above Proposition \ref{equivariantadjunction}.
Then we have 
$$
    I_{\cD}\cong I_{G}\circ I_{\cC}
    \,.
$$
In particular, $I_{\cD}(\mathbbm{1})\cong I_{G}(I_{\cC}(\mathbbm{1}))$. Thus the description of the algebra structure on $I_{G}(A)$ (for an arbitrary algebra $A$) from Section \ref{Equivariantization} provides a concrete realization of the canonical Lagrangian algebra $I_{\cD}(\mathbbm{1})$ 
which is only defined up to an isomorphism.

To compute the fusion rules, our first step is to identify $\operatorname{End}(I_{\cD}(\mathbbm{1}))$ as a vector space. 
We see 
$$
    \cZ(\cD)(I_{\cD}(\mathbbm{1}), I_{\cD}(\mathbbm{1}))\cong \cZ(\cC)^{G}(I_{G}(I_{\cC}(\mathbbm{1})), I_{G}(I_{\cC}(\mathbbm{1})) )\cong \bigoplus_{g\in G} \cZ(\cC)(I_{\cC}(\mathbbm{1}), g(I_{\cC}(\mathbbm{1})))
$$
where the first isomorphism uses the equivalence $Z(\cD)\cong Z_{\cC}(\cD)^{G}$ together with the fact that $I_{D}(\mathbbm{1})\cong I_{G}(I_{\cC}(\mathbbm{1}))$ is contained in the subcategory $Z(\cC)^{G}\subseteq Z_{\cC}(\cD)^{G}$, since $\mathbbm{1}\in \cC$. The second isomorphism uses the model for $I_{G}$ and the adjunction from Proposition \ref{equivariantadjunction}.  Let $L:=I_{\cC}(\tu)$, with multiplication $m$ and comultiplication $m^{\prime}$ as described in Equation \ref{Lagrangianalgmult}. Then using the description of the adjunction to transport the convolution and composition structures from $I_{G}(I_{\cC}(\tu))$, we have
$$
    K_{0}(\cD)\cong\bigoplus_{g\in G} Z(\cC)(L, g^{-1}(L)) \,.
$$
For $a_{g} \in Z(\cC)(L, g^{-1}(L)), b_{h}\in  Z(\cC)(L, h^{-1}(L))$
\begin{equation}
    \label{convolutionproduct}
    a_{g}\ast b_{h}:=\delta_{g,h}\ g^{-1}(m)\circ \rho^{g^{-1}}_{L,L}\circ (a_{g}\otimes b_{g})\circ m'\in Z(\cC)(L, g^{-1}(L)) \,.
\end{equation}
For the composition product, we see
\begin{equation}
    \label{compotionproduct}
    a_{g}\circ b_{h}:= \mu^{L}_{h^{-1},g^{-1}} \circ h^{-1}(a_{g})\circ b_{h}\in Z(\cC)(L, (gh)^{-1}(L))\,.
\end{equation}

\medskip

\begin{rem}Note that while we use $\circ$ for the composition product above, this is an abuse of notation, and is not actually the operation of composition of the morphisms $a_{g}$ and $b_{h}$ in the category $Z(\cC)$. Indeed this doesn't even make sense in general since they have different sources and targets. Rather, this operation corresponds to honest composition of the endomorphisms of $I_{G}(L)$ obtained by applying the adjunction from Proposition \ref{equivariantadjunction}.
\end{rem}

We now put everything together to describe an algorithm for finding the fusion coefficients of a $G$-extension of a fusion category:

\bigskip

\paragraph{\textbf{Algorithm for finding fusion rules of $G$-extension:}}

\begin{enumerate}
    \item 
 First find arbitrary basis $B_{g}$ for $V_{g}:=\cZ(\cC)(L, g^{-1}(L))$ for each $g\in G$,
 where $L=I_\cC(\tu)$ is the canonical Lagrangian algebra in $Z(\cC)$.
 \item
Compute convolution product $\ast$ (see equation \ref{convolutionproduct}) and composition product $\circ$ (see equation \ref{compotionproduct}) in terms of the basis $\bigcup_{g\in G} B_{g}$. 
\item
Find minimal projections of $V_{g}$ with respect to convolution, label them $e_{Y}$. These will correspond to simple objects in the $g$ component of the extension.
\item
Next we want to compute the $C^{Z}_{XY}$ in the sum $e_{X}\circ e_{Y}=\sum_{Z}C^{Z}_{XY}e_{Z}$. To do this, we use $(e_{X}\circ e_{Y})*e_{Z}=C^{Z}_{XY}e_{Z}$.
\item
Next, we note that for each $e_{Y}\in V_{g}$, there is a unique $e_{\bar{Y}}\in V_{g^{-1}}$ such that $C^{\mathbbm{1}}_{Y\bar{Y}}> 0$. Set $d^{+}_{X}=\sqrt{C^{\mathbbm{1}}_{Y\bar{Y}}}$.
\item
We then determine $$N^{Z}_{XY}=\left|C^{Z}_{XY}\frac{d^{+}_{Z}}{d^{+}_{X}d^{+}_{Y}}\right|\,.$$
\end{enumerate}

\subsection{Example: Fusion categories Morita equivalent to \texorpdfstring{$\Vect(A)$}{Vec(A)}}

Let $A$ be an abelian group. Then $\cZ(\operatorname{Vec}(A))\cong \operatorname{Vec}(\hat A\times A,q)$, where $q(\varphi,a)=\varphi(a)$ is the canonical quadratic form on $\hat A \times A$ and $\hat A =\Hom(A,\CC^\times)$ is the dual group \cite{EtNiOs2010}.

Fusion categories $\cC$ together with a Morita equivalence to $\operatorname{Vec}(A)$ are described by Lagrangian algebras in $Z(\operatorname{Vec}(A))$. The Lagrangian algebra is precisely $I_{\cC}(\tu)$.

Lagrangian algebras in $ \operatorname{Vec}(\hat A\times A,q)$ correspond precisely to Lagrangian subgroups $L\le (\hat{A}\times A,q)$ \cite{DaSi2017-2}. By definition, these are precisely the subgroups with $|L|=|A|$ such that $q|_{L}=1$.
By \cite[Proposition 10.3]{EtNiOs2010}, these are in bijective correspondence with subgroups $H\le A$ together with alternating  bicharacters (which are, alternatively, in bijective correspondence with elements of $H^{2}(H, \mathbbm{C}^{\times})$).

Given $H\leq A$
and $b\in \operatorname{Alt}(H\times H,\CC^\times)$
define
$$
L_{h,b} = \{ (\varphi,h)\in \hat A\times A : \varphi\restriction H=b(h,\,\cdot\,)\}
$$
and consider the Lagrangian subgroup
$$
    L_{H,b}=\bigcup_{h\in H}L_{h,b}\,.
$$
then it follows that $(H,b)\mapsto L_{H,b}$
is a one-to-one correspondence between pairs 
$(H,b)$ as above and Lagrangian subgroups 
$L\leq (\hat A\times A,q)$.
Namely, given $L\leq(\hat A\times A,q) $ Lagrangian, define $H$ by $\{1\}\times H = L\cap (\{1\}\times A)$ and 
$b(h,k) =\varphi_h(k)$ for some $(\varphi_h,h)\in L$.

Now, for arbitrary $a\in \hat{A}\times A$, let $\chi_{a}$ be the character on $\hat{A}\times A$ defined by 
$$
    \chi_{a}(b):= \frac{q(ba)}{q(b)q(a)}\,.
$$
Suppose we have a homomorphism $\pi\colon G\to O(\hat{A}\times A,q)$, i.e.\ a homomorphism $\pi\colon G\rightarrow \Aut(\hat{A}\times A)$
 such that $q(g(a))=q(a)$, where by abuse of notation we denote $g(\slot)=\pi(g)(\slot)$. Let $\omega\colon G\times G\rightarrow \widehat{A}\times A $ be a 2-cocycle with respect to this homomorphism, i.e. 
$$
    \omega_{g,hk}g(\omega_{h,k})=\omega_{g,h}\omega_{gh,k}\,.
$$
Then this data defines a braided categorical action 
$$
    \pi^{\omega}\colon\underline{G}\rightarrow \underline{\operatorname{Aut}}^{\br}_{\otimes}(\cC(\hat{A}\times A, q))\,,
$$
where the element $g$ acts by $\pi(g)$ in the obvious way as a strict monoidal functor on $\operatorname{Vec}(\widehat{A}\times A)$. We again abuse notation and use $g(\cdot)$ to refer to the functor $\pi(g)$. To define the tensorator of the categorical action we use the monoidal natural isomorphisms 
$$
    \mu^{a}_{g,h}:=\chi_{\omega_{g,h}}(gh(a))1_{gh(a)}\colon g(h(a))=gh(a)\rightarrow gh(a)
    \,.
$$
That this is a categorical action follows from the general theory of \cite{EtNiOs2010}. However, for the sake of completeness we give a direct verification.

We need to verify for all $a\in \hat{A}\times A$, $g,h,k\in G$
$$
    \mu^{a}_{gh,k}\mu^{k(a)}_{g,h}=\mu^{a}_{g,hk}g(\mu^{a}_{h,k})
$$
Using our definition of $\mu^{a}_{g,h}$, this becomes
\begin{equation}\label{tensorator}
  \chi_{\omega_{gh,k}}(ghk(a))\chi_{\omega_{g,h}}(ghk(a))=\chi_{\omega_{g,hk}}(ghk(a))\chi_{\omega_{h,k}}(hk(a))\,.  
\end{equation}
But since $g$ preserves $q$ we have $\chi_{a}(g(b))=\chi_{g^{-1}(a)}(b)$. Thus we take the left hand side of equation \ref{tensorator}, and we compute
\begin{align*}
    \chi_{\omega_{gh,k}}(ghk(a))\chi_{\omega_{g,h}}(ghk(a))&=\chi_{ghk^{-1}(\omega_{gh,k})}(a)\chi_{ghk^{-1}(\omega_{g,h})}(a)\\
    &=\chi_{ghk^{-1}(\omega_{gh,k}\omega_{g,h})}(a)\\
    &=\chi_{ghk^{-1}(\omega_{g,hk}g(\omega_{h,k}))}(a)\\
    &=\chi_{\omega_{g,hk}}(ghk(a))\chi_{\omega_{h,k}}(hk(a))\,
\end{align*}
\noindent as desired. It turns out every braided categorical action on $\operatorname{Vec}(\hat{A}\times A, q)$ is equivalent to one of this form \cite{EtNiOs2010}.
 
Let $L\le \hat{A}\times A$ be a Lagrangian subgroup, and by an abuse of notation, let $L$ also denote the corresponding Lagrangian algebra.
 
 Then as an object, $L=\bigoplus_{a\in L} a$. The multiplication is given by 
 
 $$m:=\frac{1}{\sqrt{|A|}}\bigoplus_{a,b} m_{a,b}: \bigoplus_{a,b\in L} a\otimes b \rightarrow \bigoplus_{c\in L} c,$$ where 
$$ 
    m_{a,b}=1_{a\otimes b}
    \,.
$$
The Frobenius comultiplication is defined similarly. 
\begin{thm}\label{pointedcase} Let $\pi\colon G\rightarrow \operatorname{Aut}(\hat{A}\times A, q)$ be a group homomorphism and $\omega$ an $\hat{A}\times A$-valued 2-cocycle, and consider the categorical action constructed from this data as described above. Let $L$ be a Lagrangian subgroup and set $L_{g}:=L\cap g^{-1}(L)$. Then the simple objects in the $G$-graded component of any corresponding extension of $\operatorname{Vec}(\hat{A}\times A, q)_{L}$ (if it exists) are indexed by irreducible characters $\alpha\in \widehat{L_{g}}$. For $ \alpha\in \widehat{L_{g}}, \beta\in \widehat{L_{h}}, \gamma\in \widehat{L_{gh}}$ we have
    $$
        N^{\gamma}_{\alpha \beta}=\delta_{\alpha\beta\chi_{\omega_{h^{-1},g^{-1}}}|_{L_{g}\cap L_{h}}, \gamma|_{L_{g}\cap L_{h}}} \frac{ |L_{g}\cap L_{h}|\sqrt{|A|}}{\sqrt{|L_{g}|\, |L_{h}|\, |L_{gh}|}}\,.
    $$
\end{thm}
\begin{proof}
    Note that $\cC(\hat{A}\times A, q)_{L}\cong \Vect(\hat L, \mu)$ for some 3-cocycle $\mu\in Z^{3}(\hat L, \CC^{\times})$. All these categories are pseudo-unitary, hence we can apply our algorithm to any extension.
 
    First we compute the convolution structure. A basis for $\operatorname{Vec}(\hat{A}\times A)(L, g^{-1}(L))$ is given by 
    $$
        \{ 1_{a}\}_{a\in L\cap g(L)}\,, \qquad
        1_{a}\ast 1_{b}=\frac{1}{|A|} 1_{ab}\,.
    $$
    Let $\alpha, \beta\in \widehat{L_{g}} $ be irreducible characters. Then we have the standard formula from character theory
    $$
        \sum_{a\in L_{g}} \alpha(a)\beta(a^{-1}b)=\delta_{\alpha,\beta}|L_{g}| \alpha(b)\,.
    $$
    Thus we may define
    $$
        e_{\alpha}=\frac{|A|}{|L_{g}|}\bigoplus_{a\in L_{g}} \alpha(a) 1_{a}
    $$
    and hence
    $$
        e_{\alpha}\ast e_{\beta}=\delta_{\alpha, \beta} e_{\alpha}\,.
    $$
    For $\alpha\in \widehat{L_{g}}, \beta\in \widehat{L_{h}}, \gamma\in \widehat{L_{gh}}$ we compute
    $$ 
        (e_{\alpha}\circ e_{\beta})\ast e_{\gamma}=\frac{|A|^{2}}{|L_{g}| |L_{h}| |L_{gh}| }\bigoplus_{b\in L_{gh}}\left(\sum_{a\in L_{g}\cap L_{h}} \chi_{\omega_{h^{-1},g^{-1}}}(a)\alpha(a)\beta(a)\gamma(a^{-1}b)   \right)1_{b}
        =C^{\gamma}_{\alpha \beta}e_{\gamma}\,.
    $$
    We know the coefficient of $1_{1}$ in $e_{\gamma}$ is precisely $\frac{|A|}{|L_{gh}|}$.

    Using the character formula
    $$
        \sum_{a\in L_{g}\cap L_{h} } \chi_{\omega_{h^{-1},g^{-1}}}(a)\alpha(a)\beta(a)\gamma(a^{-1}b)=|L_{g}\cap L_{h}|\ \delta_{\alpha\beta\chi_{\omega_{h^{-1},g^{-1}}}} |_{L_{g}\cap L_{h}}, \gamma|_{L_{g}\cap L_{h}}\ \gamma(b)
    $$
    and comparing coefficients of $1_{1}$ via our above expression, we obtain
    $$
        C^{\gamma}_{\alpha \beta}=\frac{|L_{g}\cap L_{h}| |A|}{|L_{g}| |L_{h}|}\ \delta_{\alpha\beta \chi_{\omega_{h^{-1},g^{-1}}}|_{L_{g}\cap L_{h}}, \gamma|_{L_{g}\cap L_{h}}}\,.
    $$
    To find the multiplicities, we see the positive dimensions satisfy 
    $$d^{+}_{\alpha}=\sqrt{\frac{|A|}{|L_{g}|}}$$ and therefore
    \begin{equation*}
        N^{\gamma}_{\alpha \beta}=\delta_{\alpha\beta \chi_{\omega_{h^{-1},g^{-1}}}|_{L_{g}\cap L_{h}}, \gamma|_{L_{g}\cap L_{h}}} \frac{ |L_{g}\cap L_{h}|\sqrt{|A|}}{\sqrt{|L_{g}| |L_{h}| |L_{gh}|}}\,.\qedhere
    \end{equation*}
\end{proof}

We remark that these formulas can be applied to derive the fusion rules for \textit{reflection fusion categories} \cite{EtGal2018} in the case when the trivial component of the category (which is an elementary abelian p-group) has trivial 3-cocycle associator.

\subsection{Example: Fusion categories with center tensor equivalent to \texorpdfstring{$\Vect(B)$}{Vect(B)}}
We can slightly generalize the former example. 
Let $A$ be an abelian group.
We want to consider the following kind of Lagrangian extensions of $A$. 
Let $B$ be another abelian group with $|B|=|A|^2$ and  $b\colon B\times B\to\CC^\times$
a bicharacter, such that $q\colon B\to \CC^\times$ defined by $q(x)=b(x,x)$ is a non-degenerate quadratic form and that there is an embedding $\hat A\hookrightarrow B$ with $q|_{\hat A}\equiv 1$. 
Note that \cite[Lemma 4.4]{LiNg2014} implies that the modular tensor category $\cC(B,q)$ is monoidally equivalent to $\Vect(B)$.
The Lagrangian subgroup $\hat A\leq B$ gives a Lagrangian algebra $L$ in $\cC(B,q)$.
We have that $\cC(B,q)_L$ is tensor equivalent to $\Vect(A,\mu)$ for some $\mu\in H^3(A,\CC^\times)$ and $\cC(B,q)$ is braided equivalent to $Z(\Vect(A,\mu))$.
For $a \in B$ let $\chi_a$ to be the character on $B$ defined by
$$
    \chi_a(g):=\frac{q(ag)}{q(a)q(g)}=b(a,g)b(g,a)\,. 
$$
Suppose as above, we have a homomorphism $\pi\colon G\to O(B,q)$ and $\omega\colon G\times G\to B$ a 2-cocycle with respect to this homomorphism as above. By replacing $\hat A \times A$ by $B$ we get a categorical action of $G$ on $\cC(B,q)$ as before. All the arguments are the same replacing $\hat A\times A$ by $B$, thus we get the slightly more general version of Theorem \ref{pointedcase}:
\begin{thm}
    \label{pointedcase2}  
    With the above notation, set  $L_{g}:=L\cap g^{-1}(L)$ and
    $A_g=A/\{a\in A~|~\ev_a|_{L_g}=\id_{L_g}\} $
   . Then the simple objects in the $G$-graded component of the corresponding extension of $\Vect(A,\mu)\cong \cC(B, q)_{L}$ (if it exists) are indexed by irreducible characters $\alpha\in A_g$. For $ \alpha\in A_g, \beta\in A_h, \gamma\in A_{gh}$ we have
    $$
        N^{\gamma}_{\alpha \beta}=\delta_{\alpha\beta\chi_{\omega_{h^{-1},g^{-1}}}|_{L_{g}\cap L_{h}}, \gamma|_{A_{g}\cap A_{h}}} \frac{ |A_{g}\cap A_{h}|\sqrt{|A|}}{\sqrt{|A_{g}|\, |A_{h}|\, |A_{gh}|}}\,.
    $$
\end{thm}
We remark that theorem applies to any $\cC$ whose center $Z(\cC)$ is tensor equivalent to $\Vect(B)$. Namely, in this case $Z(\cC)$ is braided equivalent 
to $\cC(B,q)$ where $q$ is a quadratic form on $B$ which comes from a bicharacter $b$ on $B$ by \cite[Lemma 4.4]{LiNg2014}. Then $I(\tu)\in\cC(B,q)$ gives a Lagrangian subgroup $L$ and $\cC(B,q)_L$ is tensor equivalent to $\Vect(A,\mu)$ for some $\mu\in H^3(A,\CC^\times)$, where $A=\hat L$. In particular, it applies to $\cC=\Vect(A,\mu)$ where $A$ is of odd order and $\mu$ is a ``soft'' cocycle.
Here the subgroup of ``soft'' cohomology classes is
$$ 
    H^3(A,\CC^\times)_\mathrm{ab} =
    \left\{[\omega]\in H^3(A,\CC^\times) ~\middle|~ 
    \prod_{\sigma\in S_3}\omega(\sigma(x),\sigma(y),\sigma(z))^{\operatorname{sign}\sigma} =1 \text{ for all }x,y,z\in A\right\}\,.
$$
By \cite[Corollary 3.6]{MaNg2001}, $Z(\Vect(A,\mu))$ is pointed for every 
$[\mu]\in  H^3(A,\CC^\times)_\mathrm{ab}$ thus braided equivalent to some $\cC(B,q)$.
If $A$ is odd, define a bicharacter $b$ on $B$ by 
$$b(g,h)=\left(\frac{q(gh)}{q(g)q(h)}\right)^{\frac{\operatorname{Exp}(G)+1}2}\,$$
then $q(g)=b(g,g)$. 
Thus for every $A$ odd abelian group and $[\mu]\in H^3(A,\CC^\times)_\mathrm{ab}$ the category $\Vect(A,\mu)$ arise in the above way.

\section{Examples from \texorpdfstring{$G$}{G}-crossed extensions of modular categories}

We turn our attention to the case $\cC$ is modular. Then $\cZ(\cC)$ is braided equivalent
to $\cC\boxtimes\rev{\cC}$, and the forgetful functor is the functor $X\boxtimes Y\mapsto X\otimes Y\in \cC$.
We first give a description of $L=I(\tu)$
in $\cC\boxtimes\rev{\cC}$. By \cite[Section 2.2]{KoRu2008}, we can describe the canonical Lagrangian algebra as follows:

As an object
$$
    L = \bigoplus_{X\in\Irr(\cC)} X\boxtimes \bar X\,.
$$
The multiplication is given by 
\begin{align*}
    m_{0} =\bigoplus_{X,Y,Z\in \Irr(\cC)}\sum_{i}
     \tikzmath{
         \begin{scope}[yscale=-1]
	    \draw[thick] (0,-1) node [above] {$Z$}--(0,0) .. controls (1,1) ..
         (1,1.5) node [below] {$Y$};
  		\draw[thick] (0,0) node [right] {$\,i$} .. controls (-1,1) .. 
        (-1,1.5) node [below] {$X$};
		  \mydotw{(0,0)};
		  \end{scope}
  	}
    \boxtimes
    \tikzmath{
        \begin{scope}[yscale=-1]
	    \draw[thick] (0,-1) node [above] {$\bar Z$}--(0,0) .. controls (1,1) ..
         (1,1.5) node [below] {$\bar Y$};
	  	\draw[thick] (0,0) node [right] {$\,\check i$} .. controls (-1,1) .. 
        (-1,1.5) node [below] {$\bar X$};
  		\mydotw{(0,0)};
  		\end{scope}
	  }
    \,,
\end{align*}
where $\{i\}\subseteq \cC(X\otimes Y, Z) $ consists of a basis for $\cC(X\otimes Y, Z)$, and $\{\check{i}\}\subseteq \cC(\bar{X}\otimes \bar{Y}, \bar{Z}) $ is a basis 
given
by
$$
    \tikzmath{
        \begin{scope}[yscale=-1]
	    \draw[thick] (0,-2) node [above] {$\bar Z$}--(0,0) .. controls (1,1) ..
         (1,1.5) --(1,2.5) node [below] {$\bar Y$};
	  	\draw[thick] (0,0) node [right] {$\,\check i$} .. controls (-1,1) .. 
        (-1,1.5) --(-1,2.5)node [below] {$\bar X$};
  		\mydotw{(0,0)};
  		\end{scope}
	  }
    =
    \tikzmath{
    \begin{scope}[yscale=-1]
	  	\draw[thick] (0,0) node [left] {$i^\dagger$} .. controls (-1,-1) .. 
        (-1,-1.5) arc(180:360:1) -- (1,1.5) node [below] {$\bar X$};
	    \draw[ultra thick,white,double] (0.2,-0.2).. controls (1,-1.5) .. (2,-1.5)
        arc (270:360:1);
	    \draw[thick] (-2,-3) node[above]{$\bar Z$} 
          --(-2,0) arc (180:0:1) .. controls (1,-1.5) ..
         (2,-1.5) arc (270:360:1) --(3,1.5) node [below] {$\bar Y$};
  		\mydotw{(0,0)};
  		\end{scope}
	  }\,.
$$
where $\{i^\dagger\}$ is the dual basis with respect to the composition pairing $\cC(Z,X\otimes Y)\otimes \cC(X\otimes Y, Z)\rightarrow \mathbbm{C}$, $f\otimes g\mapsto \langle f,g\rangle\in \mathbbm{C}$ where $\langle f,g\rangle 1_{Z}=g\circ f$

The unit of the algebra is given by 
$$\iota_{0}=1_{\tu\boxtimes \tu}.$$

Now, by Lemma \ref{isomorphism}, for the purposes of computing the based algebra $(\End_{\cC\boxtimes\rev{\cC}}(L),\circ)$, we can replace this algebra by any isomorphic algebra which also admits a symmetric Frobenius algebra with our desired normalization. In particular, we can choose one for which the normalized Frobenius comultiplication is easier to compute. 

Thus we consider the same object $L$, but with multiplication
\begin{align}\label{newmult}
    m&:=\frac{1}{\sqrt{\dim\cC}}\bigoplus_{X,Y,Z\in \Irr(\cC)}\frac{\sqrt{d_{X}d_{Y}}}{\sqrt{d_{Z}}}\sum_{i} i\boxtimes \check{i}
\end{align}
 
\noindent with $\check{i}$ defined as above, and unit
$$
    \iota =\sqrt{\dim(\cC)} 1_{\tu \boxtimes \tu}\,.
$$ 
Then the map 

$$\psi:=\bigoplus_{Z\in \Irr(\cC)}\sqrt{\frac{d_{Z}}{\dim(\cC)}} 1_{Z\boxtimes \bar{Z}}\in \End_{\cC \boxtimes \cC^{\text{rev}}}(L)
$$
is an automorphism of the object $A$, but satisfies $\psi \circ m=m_{0}\circ (\psi\otimes \psi)$ and $\psi\circ \iota=\iota_{0}$, and thus $(A,m, \iota)\cong (A,m_{0}, \iota_{0})$.

Furthermore, we can more easily compute the correctly normalized Frobenius comultiplication to be given by 
\begin{align}\label{newcomult} m^{\prime}:=\frac{1}{\sqrt{\dim\cC}}\bigoplus_{X,Y,Z\in \Irr(\cC)}\sqrt{\frac{d_{X}d_{Y}}{d_{Z}}}\sum_{i }i^\dagger\boxtimes \check{i}^{\dagger}
\end{align}
where $\{i^\dagger\}\subseteq \cC(Z,X\otimes Y)$ is dual to $\{i\}$ with respect to the composition pairing, and $\{\check{i}^{\dagger}\}$ is dual to $\{\check{i}\}$ with respect to the composition pairing. Defining the counit
$$
    \epsilon:=\sqrt{\dim(\cC)}1_{\tu\boxtimes \tu}
$$
\noindent we obtain a symmetric Frobenius algebra structure on $(A, m, \iota)$ with the correct normalization as desired.

We note in the modular case, the story is considerably simplified for two reasons. First, we have $(\End(A), \circ)\cong \Fun(\Irr(\cC))$ as algebras, where the latter is the algebra of complex valued functions  on the set $\Irr(\cC)$ with pointwise multiplication. The identification is via 
$$  
    f\in \Fun(\Irr(\cC))\mapsto \bigoplus_{X\in \Irr(\cC)} f(x)1_{X\boxtimes \bar{X}}\,.
$$
With this notation, an easy computation gives
\begin{equation}\label{Formula for convolution product 1}
(f*g)(Z)=\frac{1}{\dim\cC}\sum_{X,Y\in \Irr(\cC)}f(X)g(Y) \frac{d_{X} d_{Y}}{d_{Z}} N^{Z}_{XY}\,.
\end{equation}
In terms of the basis $\{1_{X\boxtimes \bar{X}}\ :\ X\in \Irr(\cC)\}$ we have 
$$ 
    1_{X\boxtimes \bar{X}}\ast 1_{Y\boxtimes \bar{Y}}=\frac{1}{\dim\cC}\sum_{Z\in \Irr(\cC)} N^{Z}_{XY} \frac{d_{X}d_{Y}}{d_{Z}} 1_{Z\boxtimes \bar{Z}}\,.
$$
We will use both expressions in the sequel based on convenience.

\begin{lem}
    \label{lem:MTCformula}
    Let $\cC$ be a modular tensor category, then 
    $$\frac{1}{\dim\cC}\sum_{X,Y\in\Irr(\cC)}N_{X,Y}^Z
    (d_V S_{X,V})(d_W S_{Y,W})=\delta_{V,W}\ d_VS_{Z,V}$$
    for all $V\in\Irr(\cC)$. 
\end{lem}
\begin{proof}
We have
$$
    \frac 1{d_W}S_{X,W}S_{Y,W}=\sum_{Z\in\Irr(\cC)}
    N_{X,Y}^ZS_{Z,W}
$$
and therefore
\begin{align*}
    \sum_{X,Y\in\Irr(\cC)}N_{X,Y}^Z(d_{V}S_{X,V})(d_{W}S_{Y,W})
    &=d_{V}d_{W}\sum_{X\in\Irr(\cC)}S_{X,V}
    \sum_{Y\in\Irr(\cC)} N_{\bar X,Z}^Y S_{Y,W}
    \\&=
    d_{V}\sum_{X\in\Irr(\cC)}S_{X,V}S_{\bar X,W} S_{Z,W}
    \\&=\delta_{V,W}\cdot \dim\cC\cdot d_{V}\ S_{Z,V}\,.
\end{align*}
In the last equality, we used the property of modular data $SCS=CS^{2}=\dim \cC \cdot I$, where $C=(\dim\cC)^{-1}\cdot S^2$ 
 is the  charge conjugation matrix 
given by $C_{X,Y} =\delta_{X, \bar{Y}}$
and $I$ is the identity matrix
 (see \cite[Theorem 3.1.7]{BaKi2001}).
\end{proof}
\begin{prop}\label{convolutionbasis}
    The set $\{e_V\}_{V\in \Irr(\cC)}$ with
    $$
         e_{V}=
        \bigoplus_{X\in \Irr(\cC)}\frac{d_V}{d_X}S_{X,V}\cdot 1_{X\boxtimes\bar X}
    $$
    forms a complete set of minimal idempotents for $(\End_{\cZ(\cC)}(L), \ast)$.
\end{prop}

\begin{proof}
Note $\{e_V\}_{V\in\Irr(\cC)}$ forms a basis of $\End_{\cZ(\cC)}(L)$
since $S$ is invertible.

Furthermore, we have 
\begin{align*}
e_{V}\ast e_{W}&=\frac{1}{\dim\cC}\sum_{X,Y,Z}\frac{d_{V}}{d_{X}}S_{X,V}\frac{d_{W}}{d_{Y}}S_{Y,W}\frac{d_{X}d_{Y}}{d_{Z}} N^{Z}_{XY} 1_{Z\boxtimes \bar{Z}}\\
&=\frac{1}{\dim\cC}\sum_{X,Y,Z}N^{Z}_{XY} (d_{V}S_{X,V})(d_{W}S_{Y,W})\frac{1}{d_{Z}}  1_{Z\boxtimes \bar{Z}}\\
&=\delta_{V,W} \sum_{Z\in \Irr(\cC)} \frac{d_{V}}{d_{Z}} S_{Z,V} 1_{Z\boxtimes \bar{Z}}=\delta_{V,W}\ e_{V}\,.\qedhere
\end{align*}
\end{proof}
As a consistency check, we compute the composition structure
\begin{align*}
    e_X\circ e_Y&=\bigoplus_{V\in\Irr(\cC)}
    \frac{d_Xd_Y}{d_V^2}S_{X,V}S_{Y,V}\cdot 1_{V\boxtimes \bar V}
    \\&=\bigoplus_{V\in\Irr(\cC)}\sum_{Z\in\Irr(\cC)}
    N_{X,Y}^Z\frac{d_Xd_Y}{d_Z}\frac{d_Z}{d_V}S_{Z,V}
    \cdot 1_{V\boxtimes \bar V}
    \\&=\sum_{Z\in\Irr(\cC)}N_{X,Y}^Z\frac{d_Xd_Y}{d_Z}\cdot
    e_Z\,.
\end{align*}

This analysis was purely of the canonical Lagrangian algebra, with no categorical action. We will now study a particular type of categorical action, which is associated to a \textit{$G$-crossed braided} extension of $\cC$ rather than an ordinary extension.

Recall for non-degenerate fusion categories that there is a monoidal functor $\pi\colon\underline{\operatorname{Aut}}^\mathrm{\br}_{\otimes}(\cC)\rightarrow \underline{\operatorname{Aut}}^\mathrm{\br}_{\otimes}(\cZ(\cC))=\underline{\operatorname{Aut}}^\mathrm{\br}_{\otimes}(\cC\boxtimes \rev\cC)$, where the braided $\cC$ autoequivalence acts on the second factor $1\boxtimes \rev\cC\subseteq \cC\boxtimes \rev\cC $. 
Furthermore, there is an equivalence $\partial\colon\underline{\operatorname{Aut}}^\mathrm{\br}_{\otimes}(\cC)\cong  \underline{\operatorname{Pic}}(\cC)$ defined via $\alpha$ induction \cite{DN}. Then by \cite{DN}, the following diagram commutes up to monoidal natural isomorphism 
$$
\begin{tikzcd}
  \underline{\operatorname{Aut}}^\mathrm{\br}_{\otimes}(\cC) \arrow{d}{\partial}\arrow{r}{\pi}
    &
   \underline{\operatorname{Aut}}^\mathrm{\br}_{\otimes}(\cZ(\cC))\arrow{d}
\\
    \underline{\operatorname{Pic}}(\cC)\arrow{r}{\text{Forget}}
    &
    \underline{\operatorname{BrPic}}(\cC)
 \end{tikzcd}
$$
where the arrow on the right is the canonical equivalence from Section \ref{extension theory}. 
 
Consider a $G$-graded extension of a modular category $\cC$. Lifts (indicated by the dotted line)
$$
    \begin{tikzcd}
         & \underline{\operatorname{Aut}}^\mathrm{\br}_{\otimes}(\cC)\cong \underline{\operatorname{Pic}}(\cC)\arrow{d} \\
         \underline{G}\arrow[ur, dashed]\arrow{r} & \underline{\operatorname{Aut}}^\mathrm{\br}_{\otimes}(\cZ(\cC))\cong \underline{\operatorname{BrPic}}(\cC)
    \end{tikzcd}
$$
of the canonically associated categorical action (indicated by the horizontal sold line)
correspond to equivalence classes of \textit{G-crossed braidings} on the $G$-extension with compatible $G$-grading \cite[Theorem 7.2]{10.1093/imrn/rnab133}.

For $g\in \underline{\operatorname{Aut}}^{\br}_{\otimes}(\cC)$, define $\Fix_{g}=\{X\in \Irr(\cC)\ :\ g(X)\cong X\}$. Choose an isomorphism $\gamma^{g}_{X}:\bar{X}\cong g^{-1}(\bar{X})$ for each $X\in \Fix_{g}=\Fix_{g^{-1}}$ (note $\bar{X}\in \Fix_{g}$ if and only if $X\in \Fix_{g}$). Then we have an isomorphism (recall that $g$ is acting only on the second factor of $\cC\boxtimes \cC^{rev}$)
\begin{align*}
    \Fun(\Fix_{g})&\cong \cZ(\cC)(L, g^{-1}(L))
    \\
    f&\mapsto \bigoplus_{X\in \Fix_{g}} f(X) \cdot 1_{X}\boxtimes \gamma^{g}_{X}\,.
\end{align*}

Our first step is to determine a formula for the convolution product on each component $\Fun(\Fix_{g})\cong \cZ(\cC)(L, g^{-1}(L))$. Utilizing our choice of $\gamma^{g}_{X}$ we can define a linear operator
\begin{align*}
    U^{g,Z}_{X,Y}\colon \cC(\bar{X}\otimes \bar{Y}, \bar{Z})&\rightarrow \cC(\bar{X}\otimes \bar{Y}, \bar{Z})
    \\
    \alpha&\mapsto (\gamma^{g}_{Z})^{-1}\cdot g^{-1}(\alpha)\cdot \rho^{g^{-1}}_{\bar{X},\bar{Y}}\cdot (\gamma^{g}_{X}\otimes \gamma^{g}_{Y})\,.
\end{align*}
where $\rho^{g^{-1}}_{\bar{X},\bar{Y}}:g^{-1}(\bar{X})\otimes g^{-1}(\bar(Y))\rightarrow g^{-1}(\bar(X)\otimes \bar(Y))$ are the structure morphisms of ``tensorator" of the monoidal functor $g^{-1}$ (see Section \ref{Equivariantization}). Then define 
$$
    T^{g,Z}_{X,Y}:=\Tr(U^{g,Z}_{X,Y})\,.
$$

Applying the formula for convolution \eqref{convolutionproduct} together with Equations \ref{newmult} and \ref{newcomult} for the multiplication and comultiplication on $L$, we obtain the following proposition.
\begin{prop}
    \label{prop:convolution}
    Let $\{\gamma^{g}_{X}: \bar{Z}\rightarrow g^{-1}(\bar{X})\}_{X\in \Fix_{g}}$ be a family of isomorphisms and $X,Y\in \Fix_{g}$. Then 
    $$
        (1_{X}\boxtimes\gamma^{g}_{X})\ast( 1_{Y}\boxtimes \gamma^{g}_{Y})=\frac{1}{\dim\cC}\sum_{Z\in \Fix_{g}} \frac{d_{X}d_{Y}}{d_{Z}}T^{g,Z}_{X,Y} \cdot 1_{Z}\boxtimes\gamma^{g}_{Z}\,.
    $$ 
\end{prop}

Note that both $U^{g,Z}_{X,Y}$ and $T^{g,Z}_{X,Y}$ depend on the choices of the $\gamma^{g}_{X}$. However, for a fixed $g$ for each $X$, the possible choices of isomorphism $\bar{X}\cong g^{-1}(\bar{X})$ form a torsor over $\mathbbm{C}^{\times}$. We see modifying a given $\gamma^{g}_{X}$ by a scalar $\lambda_{X}$, the quantity $U^{g,Z}_{X,Y}$ (and hence $T^{g,Z}_{X,Y}$) are multiplied by the factor $\frac{\lambda_{X} \lambda_{Y}}{\lambda_{Z}}$. The basis elements $1_{X}\boxtimes \gamma^{g}_{X}$ are multiplied by the scalar $\lambda_{X}$, and we see that the modifications cancel.

\subsection{Cyclic permutation actions}
Let $\cC$ be a modular tensor category and
$G\hookrightarrow S_n$ a $G$-set. 
More precisely, we choose a natural number $n$, a finite group $G$, and a faithful action of $G$ on the set $\{1,\ldots, n\}$ which we can naturally see as a monomorphism $G\hookrightarrow S_n$. 
Then there is a strict action of $G$ on $\cC^{\boxtimes n}$ by permutations. 
We denote a $G$-crossed braided extension of $\cC^{\boxtimes n}$ by $\cC\wr G$
 as in \cite{Tu2010} 
 noting that it is not necessarily unique, but always exists by \cite{GaJo2018} (see Remark \ref{sphericalremark}). Delaney has recently given an algorithm for computing the fusion rules of arbitrary permutation crossed extensions \cite{D2019}. In this section we apply our methods to obtain a closed formula in the special case of the standard cyclic subgroup  $\ZZ/n\ZZ \cong \langle (12\cdots n)\rangle \leq S_n$,
$1\mapsto (12\cdots n)$.

Suppose $g\in \mathbbm{Z}/n\mathbbm{Z}$. Let us examine the equivalence class of simple objects in $\mathcal{C}^{\boxtimes n}$ invariant under $g$. Note that for any set with a $\mathbbm{Z}/n\mathbbm{Z}$ action, an element is left fixed by $g$ if and only it is fixed by the entire cyclic subgroup generated by $g$. Let $o(g)$ denote the order of $g$. We define the \textit{co-order} of $g$ by $c(g):=\frac{n}{o(g)}$. Let $1$ denote the standard generator of  $\mathbbm{Z}/n\mathbbm{Z}$ (we will use additive notation here for finite cyclic groups). From the elementary theory of finite cyclic groups $\langle g \rangle=\langle c(g)\rangle $. Thus an element is fixed by $g$ if and only if it is fixed by $c(g)$. Since $c(g)$ acts on $X=X_{1}\boxtimes \dots \boxtimes X_{n}$ by $$c(g)(X)=X_{n-c(g)+1}\boxtimes X_{n}\boxtimes X_{1}\boxtimes \dots \boxtimes X_{n-c(g)},$$

then (since $n-c(g)=(o(g)-1)c(g)$) then $c(g)(X)=X$ if and only if $X_{i}=X_{kc(g)+i}$ for $1\le i\le c(g)$ and $0\le k\le o(g)-1$. We conclude that the simple object $X$ fixed by $g$ up to isomorphism are of the form

$$X=(X_{1}\boxtimes X_{2}\boxtimes \dots \boxtimes X_{c(g)})^{\boxtimes o(g)}$$

\noindent where $X_{1}, \cdots, X_{c(g)}\in \Irr(\cC)$ are arbitrary. Then we have the following claim:

\begin{lem} The set of minimal convolution idempotents is given by $\{f_{g, X}\}_{X\in \Irr(\cC^{\boxtimes m})}$, where 
$$
    f_{g,X=X_1\boxtimes\cdots\boxtimes X_m}=
    \bigoplus_{Y=Y_1\boxtimes \cdots \boxtimes Y_m}
    \frac{d_X(\dim\cC)^{n-c(g)}}{d_Y^{o(g)}}
    \underbrace{\prod_{i=1}^{c(g)}
    S_{X_i,Y_i}}_{=S_{X,Y}}\cdot
    1_{Y^{\boxtimes o(g)}\boxtimes \bar Y^{\boxtimes o(g)}}\,.
$$
\end{lem}
\begin{proof}
We may assume $g$ is a generator. Otherwise replace $\cC$ by $\cC^{\boxtimes c(g)}$ and replacing $n$ by $o(g)$. Then as above $X=X^{\boxtimes n}_{1}$ and $Y=Y^{\boxtimes n}_{1}$ for $X_{1}, Y_{1}\in \cC$. Since the action is strict, we are free to choose $\gamma^{g}_{X}=1_{\bar{X}}$ for all simples $X\in \Fix_{g}$. Applying Proposition \ref{prop:convolution} to our situation we have

$$
1_{X\boxtimes \bar{X}}\ast 1_{Y\boxtimes \bar{Y}}=\frac{1}{\dim\cC^{\boxtimes n}}\sum_{Z=Z^{\boxtimes n}_{1}\in \Fix_{g}} \frac{d_{X}d_{Y}}{d_{Z}}N^{Z_1}_{X_1,Y_{1}} \cdot 1_{Z\boxtimes \bar{Z}}\,
$$
where $Z=Z^{\boxtimes n}_{1}$. This follows since the term $T^{g,Z}_{X,Y}$ from Proposition \ref{prop:convolution} is computing the trace of the cyclic translate permutation on the space $\cC^{\boxtimes n}(X\otimes Y,Z)=\cC(X_{1}\otimes Y_{1}, Z_{1})^{\otimes n}$. 
But choosing a basis $B\subseteq  \cC(X_{1}\otimes Y_{1}, Z_{1})^{\otimes n}$ the set 
$B^{n}:=\{\beta_{1}\otimes \cdots\otimes \beta_{n}\ :\ \beta_{i}\in B\}$ is a basis for the subspace of $\cC^{\boxtimes n}(X\otimes Y,Z)$ invariant under the permutation action. Thus the trace term $T^{g,Z}_{X,Y}$ is precisely the number of basis elements of $B^{n}$ fixed by $g$. Such an element is completely determined by its first entry $\beta_{1}$, and thus $T^{g,Z}_{X,Y}=\dimCC\cC(X_{1}\otimes Y_{1}, Z_{1})=N^{Z_{1}}_{X_{1},Y_{1}}$.

Expanding further, we see

$$
1_{X\boxtimes \bar{X}}\ast 1_{Y\boxtimes \bar{Y}}=\frac{1}{(\dim\cC)^{n}}\sum_{Z=Z^{\boxtimes n}_{1}\in \Fix_{g}} \frac{d^{n}_{X_1}d^{n}_{Y_1}}{d^{n}_{Z}}N^{Z_1}_{X_1,Y_{1}} \cdot 1_{Z\boxtimes \bar{Z}}\,
$$

Then we have
\begin{align*}
    f_{g,X}
    \ast
    f_{g,Y}
    &=\sum_{U,V,W\in \Irr(\cC)}\frac{d_Xd_Y}{d_U^nd_V^n}S_{X,U}S_{Y,V}
    \frac{(\dim\cC)^{2n-2}}{(\dim\cC)^{n}}
    \frac{d_U^nd_V^n}{d_W^n}
    N_{U,V}^W 
    \cdot
    1_{W^{\boxtimes n}\boxtimes\bar W^{\boxtimes n}}
    \\
        &=\sum_{U,V,W\in \Irr(\cC)}\frac{d_Xd_Y(\dim\cC)^{n-2}}{d_W^n}S_{X,U}S_{Y,V}
    N_{U,V}^W 
    \cdot
    1_{W^{\boxtimes n}\boxtimes\bar W^{\boxtimes n}}
      \\
        &=\sum_{W\in \Irr(\cC)}\frac{d_X(\dim\cC)^{n-1}}{d_W^n} \delta_{X,Y} S_{W,X}
    \cdot
    1_{W^{\boxtimes n}\boxtimes\bar W^{\boxtimes n}}
      \\
        &= \delta_{X,Y}  f_{g,X}\,.\qedhere
\end{align*}
\end{proof}

\begin{thm}\label{fusioncyclic}
    Let $\cC$ be a modular tensor category. 
    Consider a spherical $\ZZ/n\ZZ$-crossed braided permutation extension
    $$\cC\wr \ZZ/n\ZZ=
    \bigoplus_{g\in\ZZ/n\ZZ}(\cC\wr \ZZ/n\ZZ)_g$$ of $\cC^{\boxtimes n}$. 
    Then $\Irr((\cC\wr \ZZ/n\ZZ)_g)=
    \{(-g,X)\}_{X\in\Irr(\cC^{\boxtimes c(g)})}$ with $g\in\ZZ/n\ZZ$
    and fusion rules
    are given by
    \begin{align*}
    N^{(g+h,Z)}_{(g,X),(h,Y)}
    &=
    {}^{k} N_{X'_1,\ldots,X'_{{c(g)}/{p}},
    Y'_1,\ldots,Y'_{{c(h)}/{p}}}^{(p);Z'_1,\ldots, Z'_{{c(g+h)}/{p}}}\,,
    \qquad \text{where} \quad k=\frac{n-c(g)-c(h)-c(g+h)}{2p}+1
    \,.
\end{align*}
Here $g,h\in\ZZ/n\ZZ$,
$X\in\cC^{\boxtimes c(g)}$,
$Y\in\cC^{\boxtimes c(h)}$,
$Z\in\cC^{\boxtimes c(g+h)}$,
$p=\mathop{\mathrm{gcd}}(c(g),c(h))$, and ${}^k N^{(p)}$ indicate the (higher) fusion matrices (see  \eqref{definefusioncoeff}) of $\cC^{\boxtimes p}$.
Furthermore, 
$X'_i,Y'_i,Z'_i\in\cC^{\boxtimes p}$ with $X=X'_1\boxtimes \cdots \boxtimes X'_{c(g)/p}$,
$Y=Y'_1\boxtimes \cdots \boxtimes Y'_{c(h)/p}$, and 
$Z=Z'_1\boxtimes \cdots \boxtimes Z'_{c(h)/p}$.
\end{thm}
\begin{proof}
    By Proposition \ref{genverlinde}
    the statement can be written in terms of $S$-matrices $S^{(p)}$ of 
    $\cC^{\boxtimes p}$ as
    \begin{align}
    &N^{(g+h,Z)}_{(g,X),(h,Y)}\nonumber \\
    \\
    &=
    \sum_{W\in \Irr(\cC^{\boxtimes p})}\left(\prod^{\frac{c(g)}{p}}_{i=1} \frac{S^{(p)}_{X^{\prime}_{i},W}}{S^{(p)}_{\tu, W}}\right) \left(\prod^{\frac{c(h)}{p}}_{i=1} \frac{S^{(p)}_{Y^{\prime}_{i},W}}{S^{(p)}_{\tu, W}}\right) \left(\prod^{\frac{c(g+h)}{p}}_{i=1} \frac{S^{(p)}_{Z^{\prime}_{i},W}}{S^{(p)}_{\tu, W}}\right) \left(\frac{\sqrt{\dim\cC}^{p}}{S^{(p)}_{\tu, W}}\right)^{\frac{n-c(g)-c(h)-c(g+h)}{p}}\,.\label{eq:Formula}
    \end{align}
  We first note that we only need to prove the case 
  $p=1$. 
  If $p>1$ then the fusion rules factor through a $\ZZ/\frac np\ZZ$-cyclic permutation extension of $\cC^{\boxtimes p}$ and the formula  is obtained from the $p'=1$ formula by considering
  $n'=n/p$, $\cC'=\cC^{\boxtimes p}$,
  $g'=g/p$, and $h'=h/p$.
  
  Now, let $g,h\in G$ such that $\operatorname{gcd}(c(g),c(h))=1$
  and let $f_{g,X}$ with
  $X\in \Irr(\cC^{\boxtimes c(g)})$ and
  $f_{h,Y}$ with $Y\in \Irr(\cC^{\boxtimes c(h)})$ minimal convolution idempotents.
    Then
    $$
        f_{g,X}\circ f_{h,Y}=\bigoplus_{W\in\Irr\cC}\frac{d_{X}d_{Y}}{d^{{2n}}_{W}}(\dim\cC)^{2n-c(g)-c(h)} S_{X,W^{\boxtimes {c(g)}}} S_{Y, W^{\boxtimes {c(h)}}}
        \cdot
        1_{W^{\boxtimes n}\boxtimes \bar{W}^{\boxtimes n}}\,.
    $$
    Here we have used the fact that the only terms from $f_{g,X}$ and $f_{h,Y}$ that contribute to the composition are the coefficients of $1_{R\boxtimes \bar{R}}$, where $R\in \Irr(\cC^{\boxtimes n})$ is of the form $R=(R_{1}\boxtimes R_{2}\boxtimes \dots \boxtimes R_{c(g)})^{\boxtimes o(g)}=(R^{\prime}_{1}\boxtimes R^{\prime}_{2}\boxtimes \dots \boxtimes R^{\prime}_{c(h)})^{\boxtimes o(h)}$, with $R_{i}, R^{\prime}_{i}\in \Irr(\cC)$. However, since $p=(c(g),c(h))=1$, we must have $R=W^{\boxtimes n}$ for $W\in \Irr(\cC).$, which gives the above expression.

    Now let $Z=Z_{1}\boxtimes \dots \boxtimes Z_{c(g+h)}$, so that $f_{g+h, Z}$ is a minimal convolution idempotent. 
    Then we have
    \begin{align*}
        &(f_{g,X}\circ f_{h,Y})\ast f_{g+h,Z} \\
        &=\bigoplus_{W,U\in \Irr(\cC)} \frac{d_{X}d_{Y}d_{Z}}{d^{2n}_{W}d^{n}_{U}}(\dim\cC)^{3n-c(g)-c(h)-c(g+h)} S_{X,W^{\boxtimes {c(g)}}} S_{Y, W^{\boxtimes {c(h)}}}
        S_{Z,U} (1_{W^{\boxtimes n}\boxtimes \bar W^{\boxtimes n}}\ast 1_{U^{\boxtimes n} \boxtimes \bar U^{\boxtimes n}})
        \\
        &=\bigoplus_{V} \sum_{W,U\in \Irr(\cC)} \frac{d_{X}d_{Y}d_{Z}}{d^{2n}_{W}d^{\frac{n}{c(g+h)}}_{U}}(\dim\cC)^{l}S_{X,W^{\boxtimes {c(g)}}} S_{Y, W^{\boxtimes {c(h)}}}S_{Z,U} \frac{d^{{n}}_{W} d^{o(g+h)}_{U}}{d^{o(g+h)}_{V}}N^{V}_{W^{\boxtimes c(g,h)},U}\cdot 1_{V^{\boxtimes}\boxtimes\bar{V}^{\boxtimes}}\,,
    \end{align*}
     where $l:=2n-c(g)-c(h)-c(g+h)$. Comparing coefficients for $V=1$, we obtain the equation
    \begin{align*}
        &\sum_{W,U\in \Irr(\cC)} \frac{d_{X}d_{Y}d_{Z}}{d^n_{W}}(\dim\cC)^{2n-c(g)-c(h)-c(g+h)}S_{X,W^{\boxtimes {c(g)}}} S_{Y, W^{\boxtimes {c(h)}}} S_{Z,\bar{W}^{c(g+h)}}
    \\&=C^{(g+h,Z)}_{(g,X),(h,Y)} d^{2}_{Z}(\dim\cC)^{n-c(g+h)}
    \end{align*}
 hence
    $$
        C^{(g+h,Z)}_{(g,X),(h,Y)}=\sum_{W\in \Irr(\cC)} \frac{d_{X}d_{Y}}{d_{Z}d^n_{W}}(\dim\cC)^{n-c(g)-c(h)}
        S_{X,W^{\boxtimes {c(g)}}} S_{Y, W^{\boxtimes {c(h)}}}S_{Z,\bar{W}^{c(g+h)}}\,.
    $$
    We now see
    $$
        d^{+}_{f_{g,X}}=d_{X}(\dim\cC)^{\frac{n-c(g)}{2}}\,.
    $$
    Hence normalizing we obtain
    $$
        N^{(g+h,Z)}_{(g,X),(h,Y)}=\sum_{W\in \Irr(\cC)} \frac{\sqrt{\dim\cC}^{n-c(g)-c(h)-c(g+h)}}{d^n_{W}}S_{X,W^{\boxtimes {c(g)}}} S_{Y, W^{\boxtimes {c(h)}}}S_{Z,\bar{W}^{c(g+h)}}\,.
    $$
    The right hand side factorizes into the expression \eqref{eq:Formula}.

\end{proof}
Note that in the case that  the genus
${\textstyle\frac{n-c(g)-c(h)-c(g+h)}{2p}+1}$ vanishes, we have that
$$
N^{(g+h,Z)}_{(g,X),(h,Y)}=
\dimCC\cC^{\boxtimes p}
\big(
X'_1\otimes \cdots \otimes X'_\frac{c(g)}{p}\otimes 
Y'_1\otimes \cdots\otimes Y'_\frac{c(h)}{p},
Z'_1\otimes \cdots \otimes Z'_\frac{c(g+h)}{p}
\big)
$$
for example, we recover a well-known special 
case 
$$
N^{(0,Z_1\boxtimes\cdots\boxtimes Z_n)}_{(g,\tu),(-g, \tu)}
= N_{Z_1,\ldots,Z_n}
$$
first observed for multiplicities in $n$-interval inclusions 
\cite{KaLoMg2001} and later for fusion rules in cyclic permutations \cite{LoXu2004} of conformal nets.

\bigskip

In order to relate the expression to the \emph{generalized Verlinde formula} we need that $g$ is an integer 
and therefore require the exponent ${\textstyle\frac{n-c(g)-c(h)-c(g+h)}{p}}$ to be even. 
As a consistency check, we verify the following lemma which verifies that this is indeed the case.
\begin{lem}
    $\frac{n-(m,n)-(k,n)-(m+k,n)}{((m,n),(k,n))}$ is even.
\end{lem}

\begin{proof}
First we claim the set of numbers $\{\frac{(m,n)}{((m,n), (k,n))},\frac{(k,n)}{((m,n), (k,n))},\frac{(m+k,n)}{((m,n), (k,n))}\}$ is pairwise coprime. Note that $\frac{(k,n)}{((m,n),(k,n))}$ and $\frac{(m,n)}{((m,n),(k,n))}$ are coprime. Suppose $l|\frac{(m+k,n)}{((m,n),(k,n))}$ and $l|\frac{(m,n)}{((m,n),(k,n))}$. Then 
$((m,n),(k,n))l|m+k$ and $((m,n),(k,n))l|m$ thus $((m,n),(k,n))l|k$. Because also $((m,n),(k,n))l|n$ we have $((m,n),(k,n))l|(k,n)$.
Thus $l|\frac{(k,n)}{{((m,n),(k,n))}}$. But $\frac{(k,n)}{{((m,n),(k,n))}}$ is coprime to $l|\frac{(m,n)}{{((m,n),(k,n))}}$, thus $l$ must be $1$. A similar argument applies switching $m$ and $k$, and we get that these numbers are coprime.

\medskip

We break the rest of the proof into parity cases.
\begin{enumerate}
    \item Suppose $\frac{n}{((m,n), (k,n))}$ is odd. Then since the other three terms in the expression must divide this one, they are also odd so we see that the whole expression is even (odd-odd-odd-odd=even).
    \item
    Suppose $\frac{n}{((m,n), (k,n))}$ is even and at least one of the other three terms is even. Then the other 2 must be odd since they are pairwise co-prime. Thus the whole expression is even, since (even-even-odd-odd)=even.
    \item
    Suppose $\frac{n}{((m,n), (k,n))}$ is even, but the other three terms are all odd. We claim this is not possible. Indeed, suppose $\frac{(m,n)}{((m,n), (k,n))},\frac{(k,n)}{((m,n), (k,n))}$ are both odd. Then $\frac{m}{((m,n), (k,n))},\frac{k}{((m,n), (k,n))}$ must both be odd hence $\frac{m+k}{(((m,n), (k,n))}$ must be even, thus $2| \frac{m+k}{(((m,n), (k,n))}$ and $2| \frac{n}{(((m,n), (k,n))}$ by hypothesis, so $\frac{(m+k, n)}{(((m,n), (k,n))}$ is even. Thus this cannot occur.
\end{enumerate} 
\end{proof}

\begin{rem}\label{sphericalremark} There is a subtle point about our argument. The results of \cite{GaJo2018} guarantee the existence of such extensions as a G-crossed braided fusion category (indeed, in the cyclic case this follows from \cite{EtNiOs2010} since $H^{4}(\ZZ/n\ZZ, \mathbbm{C}^{\times})=1$), but make no claims as to the existence of a spherical structure.  The computation of the fusion rules explicitly assumes the existence of a spherical structure on the extension $\cC\wr \ZZ/n\ZZ$ of $\cC^{\boxtimes n}$. It is widely expecting that there exist a spherical structure on these extensions provided $\cC$ itself admits one. If we assume $\cC$ is pseudo-unitary, then by Proposition \ref{pseudounitaryext}, the extension will be and hence admits a spherical structure. 
\end{rem}

\bibliography{bibliography.bib}

\def\cprime{$'$}\newcommand{\noopsort}[1]{}
\begin{bibdiv}
\begin{biblist}

\bib{Ban2002}{article}{
      author={Bantay, P.},
       title={Permutation orbifolds},
        date={2002},
        ISSN={0550-3213},
     journal={Nuclear Phys. B},
      volume={633},
      number={3},
       pages={365\ndash 378},
         url={https://doi.org/10.1016/S0550-3213(02)00198-0},
      review={\MR{1910268}},
}

\bib{Ba16}{article}{
      author={Bartlett, Bruce},
       title={Fusion categories via string diagrams},
        date={2016},
        ISSN={0219-1997},
     journal={Commun. Contemp. Math.},
      volume={18},
      number={5},
       pages={1550080, 39},
         url={https://doi.org/10.1142/S0219199715500807},
      review={\MR{3523185}},
}

\bib{BaBoChWa2014}{article}{
      author={Barkeshli, Maissam},
      author={Bonderson, Parsa},
      author={Cheng, Meng},
      author={Wang, Zhenghan},
       title={Symmetry fractionalization, defects, and gauging of topological
  phases},
        date={2014},
     journal={arXiv preprint arXiv:1410.4540},
}

\bib{BiDa2018}{misc}{
      author={Bischoff, Marcel},
      author={Davydov, Alexei},
       title={Hopf algebra actions in tensor categories},
        date={2018},
        note={To appear in Trans. Groups},
}

\bib{BaDoSPVi2015}{article}{
      author={Bartlett, Bruce},
      author={Douglas, Christopher~L},
      author={Schommer-Pries, Christopher~J},
      author={Vicary, Jamie},
       title={Modular categories as representations of the 3-dimensional
  bordism 2-category},
        date={2015},
     journal={arXiv preprint arXiv:1509.06811},
}

\bib{BaFRS}{article}{
      author={Barmeier, Till},
      author={Fuchs, J\"{u}rgen},
      author={Runkel, Ingo},
      author={Schweigert, Christoph},
       title={Module categories for permutation modular invariants},
        date={2010},
        ISSN={1073-7928},
     journal={Int. Math. Res. Not. IMRN},
      number={16},
       pages={3067\ndash 3100},
         url={https://doi.org/10.1093/imrn/rnp235},
      review={\MR{2673719}},
}

\bib{BoHaSch}{article}{
      author={Borisov, L.},
      author={Halpern, M.~B.},
      author={Schweigert, C.},
       title={Systematic approach to cyclic orbifolds},
        date={1998},
        ISSN={0217-751X},
     journal={Internat. J. Modern Phys. A},
      volume={13},
      number={1},
       pages={125\ndash 168},
      review={\MR{1606821}},
}

\bib{Bi2016}{article}{
      author={Bischoff, Marcel},
       title={Generalized orbifold construction for conformal nets},
        date={2017},
        ISSN={0129-055X},
     journal={Rev. Math. Phys.},
      volume={29},
      number={1},
       pages={1750002, 53},
         url={http://dx.doi.org/10.1142/S0129055X17500027},
      review={\MR{3595480}},
}

\bib{BaJiQi2013}{article}{
      author={Barkeshli, Maissam},
      author={Jian, Chao-Ming},
      author={Qi, Xiao-Liang},
       title={Twist defects and projective non-abelian braiding statistics},
        date={2013Jan},
     journal={Phys. Rev. B},
      volume={87},
       pages={045130},
         url={https://link.aps.org/doi/10.1103/PhysRevB.87.045130},
}

\bib{BaKi2001}{book}{
      author={Bakalov, Bojko},
      author={Kirillov, Alexander, Jr.},
       title={Lectures on tensor categories and modular functors},
      series={University Lecture Series},
   publisher={American Mathematical Society, Providence, RI},
        date={2001},
      volume={21},
        ISBN={0-8218-2686-7},
      review={\MR{1797619}},
}

\bib{BiKaLoRe2014-2}{book}{
      author={Bischoff, Marcel},
      author={Kawahigashi, Yasuyuki},
      author={Longo, Roberto},
      author={Rehren, Karl-Henning},
       title={Tensor categories and endomorphisms of von {N}eumann
  algebras---with applications to quantum field theory},
      series={Springer Briefs in Mathematical Physics},
   publisher={Springer},
        date={2015},
      volume={3},
        ISBN={978-3-319-14300-2; 978-3-319-14301-9},
         url={http://dx.doi.org/10.1007/978-3-319-14301-9},
      review={\MR{3308880}},
}

\bib{BuNa}{article}{
      author={Burciu, Sebastian},
      author={Natale, Sonia},
       title={Fusion rules of equivariantizations of fusion categories},
        date={2013},
        ISSN={0022-2488},
     journal={J. Math. Phys.},
      volume={54},
      number={1},
       pages={013511, 21},
         url={https://doi.org/10.1063/1.4774293},
      review={\MR{3059899}},
}

\bib{BarSchweig}{article}{
      author={Barmeier, T.},
      author={Schweigert, C.},
       title={A geometric construction for permutation equivariant categories
  from modular functors},
        date={2011},
        ISSN={1083-4362},
     journal={Transform. Groups},
      volume={16},
      number={2},
       pages={287\ndash 337},
         url={https://doi.org/10.1007/s00031-011-9132-y},
      review={\MR{2806495}},
}

\bib{BW}{article}{
      author={Barrett, John~W.},
      author={Westbury, Bruce~W.},
       title={Spherical categories},
        date={1999},
        ISSN={0001-8708},
     journal={Adv. Math.},
      volume={143},
      number={2},
       pages={357\ndash 375},
         url={https://doi.org/10.1006/aima.1998.1800},
      review={\MR{1686423}},
}

\bib{D2019}{article}{
      author={Delaney, Colleen},
       title={{Fusion rules for permutation extensions of modular tensor
  categories}},
        date={2019},
        note={Preprint},
}

\bib{DN}{article}{
      author={Davydov, Alexei},
      author={Nikshych, Dmitri},
       title={The {P}icard crossed module of a braided tensor category},
        date={2013},
        ISSN={1937-0652},
     journal={Algebra Number Theory},
      volume={7},
      number={6},
       pages={1365\ndash 1403},
         url={https://doi.org/10.2140/ant.2013.7.1365},
      review={\MR{3107567}},
}

\bib{DaPa2008}{article}{
      author={Day, Brian},
      author={Pastro, Craig},
       title={Note on {F}robenius monoidal functors},
        date={2008},
        ISSN={1076-9803},
     journal={New York J. Math.},
      volume={14},
       pages={733\ndash 742},
         url={http://nyjm.albany.edu:8000/j/2008/14_733.html},
      review={\MR{2465800}},
}

\bib{DaSi2017-2}{article}{
      author={Davydov, Alexei},
      author={Simmons, Darren~A.},
       title={Third cohomology and fusion categories},
        date={2018},
        ISSN={1532-0073},
     journal={Homology Homotopy Appl.},
      volume={20},
      number={1},
       pages={275\ndash 302},
         url={https://doi.org/10.4310/HHA.2018.v20.n1.a17},
      review={\MR{3775361}},
}

\bib{EtGal2018}{article}{
      author={Etingof, Pavel},
      author={Galindo, C\'{e}sar},
       title={Reflection fusion categories},
        date={2018},
        ISSN={0021-8693},
     journal={J. Algebra},
      volume={516},
       pages={172\ndash 196},
  url={https://doi-org.proxy.lib.ohio-state.edu/10.1016/j.jalgebra.2018.09.006},
      review={\MR{3863475}},
}

\bib{EtGeNiOs2015}{book}{
      author={Etingof, Pavel},
      author={Gelaki, Shlomo},
      author={Nikshych, Dmitri},
      author={Ostrik, Victor},
       title={Tensor categories},
      series={Mathematical Surveys and Monographs},
   publisher={American Mathematical Society, Providence, RI},
        date={2015},
      volume={205},
        ISBN={978-1-4704-2024-6},
         url={http://dx.doi.org/10.1090/surv/205},
      review={\MR{3242743}},
}

\bib{EdJoPl2018}{article}{
      author={Edie-Michell, Cain},
      author={Jones, Corey},
      author={Plavnik, Julia},
       title={Fusion rules for {$\mathbb{Z}/ 2\mathbb{Z}$} permutation
  gauging},
        date={2018},
     journal={arXiv preprint arXiv:1804.01657},
}

\bib{EtNiOs2005}{article}{
      author={Etingof, Pavel},
      author={Nikshych, Dmitri},
      author={Ostrik, Viktor},
       title={On fusion categories},
        date={2005},
        ISSN={0003-486X},
     journal={Ann. of Math. (2)},
      volume={162},
      number={2},
       pages={581\ndash 642},
         url={http://dx.doi.org/10.4007/annals.2005.162.581},
      review={\MR{2183279 (2006m:16051)}},
}

\bib{EtNiOs2010}{article}{
      author={Etingof, Pavel},
      author={Nikshych, Dmitri},
      author={Ostrik, Victor},
       title={Fusion categories and homotopy theory},
        date={2010},
        ISSN={1663-487X},
     journal={Quantum Topol.},
      volume={1},
      number={3},
       pages={209\ndash 273},
         url={http://dx.doi.org/10.4171/QT/6},
        note={With an appendix by Ehud Meir},
      review={\MR{2677836}},
}

\bib{FuRuSc2002}{article}{
      author={Fuchs, Jürgen},
      author={Runkel, Ingo},
      author={Schweigert, Christoph},
       title={{T{FT} construction of {RCFT} correlators. {I}. {P}artition
  functions}},
        date={2002},
        ISSN={0550-3213},
     journal={Nuclear Phys. B},
      volume={646},
      number={3},
       pages={353–497},
         url={http://dx.doi.org/10.1016/S0550-3213(02)00744-7},
      review={\MR{1940282 (2004c:81244)}},
}

\bib{GaJo2018}{article}{
      author={Gannon, Terry},
      author={Jones, Corey},
       title={Vanishing of {C}ategorical {O}bstructions for {P}ermutation
  {O}rbifolds},
        date={2019},
        ISSN={0010-3616},
     journal={Comm. Math. Phys.},
      volume={369},
      number={1},
       pages={245\ndash 259},
         url={https://doi.org/10.1007/s00220-019-03288-9},
      review={\MR{3959559}},
}

\bib{GeNaNi2009}{article}{
      author={Gelaki, Shlomo},
      author={Naidu, Deepak},
      author={Nikshych, Dmitri},
       title={Centers of graded fusion categories},
        date={2009},
        ISSN={1937-0652},
     journal={Algebra Number Theory},
      volume={3},
      number={8},
       pages={959\ndash 990},
         url={http://dx.doi.org/10.2140/ant.2009.3.959},
      review={\MR{2587410}},
}

\bib{HuLep2013}{article}{
      author={Huang, Yi-Zhi},
      author={Lepowsky, James},
       title={Tensor categories and the mathematics of rational and logarithmic
  conformal field theory},
        date={2013},
        ISSN={1751-8113},
     journal={J. Phys. A},
      volume={46},
      number={49},
       pages={494009, 21},
  url={https://doi-org.proxy.lib.ohio-state.edu/10.1088/1751-8113/46/49/494009},
      review={\MR{3146015}},
}

\bib{10.1093/imrn/rnab133}{article}{
      author={Jones, Corey},
      author={Morrison, Scott},
      author={Penneys, David},
      author={Plavnik, Julia},
       title={{Extension Theory for Braided-Enriched Fusion Categories}},
        date={202107},
        ISSN={1073-7928},
     journal={International Mathematics Research Notices},
  eprint={https://academic.oup.com/imrn/advance-article-pdf/doi/10.1093/imrn/rnab133/38851094/rnab133.pdf},
         url={https://doi.org/10.1093/imrn/rnab133},
        note={rnab133},
}

\bib{MR0360749}{inproceedings}{
      author={Kelly, G.~M.},
       title={Doctrinal adjunction},
        date={1974},
   booktitle={Category {S}eminar ({P}roc. {S}em., {S}ydney, 1972/1973)},
       pages={257\ndash 280. Lecture Notes in Math., Vol. 420},
      review={\MR{0360749}},
}

\bib{Ki2002}{article}{
      author={Kirillov, Alexander, Jr.},
       title={Modular categories and orbifold models},
        date={2002},
        ISSN={0010-3616},
     journal={Comm. Math. Phys.},
      volume={229},
      number={2},
       pages={309\ndash 335},
         url={https://doi-org.proxy.lib.ohio-state.edu/10.1007/s002200200650},
      review={\MR{1923177}},
}

\bib{Kit06}{article}{
      author={Kitaev, Alexei},
       title={Anyons in an exactly solved model and beyond},
        date={2006},
        ISSN={0003-4916},
     journal={Ann. Physics},
      volume={321},
      number={1},
       pages={2\ndash 111},
         url={https://doi.org/10.1016/j.aop.2005.10.005},
      review={\MR{2200691}},
}

\bib{KiBa2010}{article}{
      author={Kirillov~Jr, Alexander},
      author={Balsam, Benjamin},
       title={Turaev-{V}iro invariants as an extended tqft},
        date={2010},
     journal={arXiv preprint arXiv:1004.1533},
}

\bib{KaLoMg2001}{article}{
      author={Kawahigashi, Y.},
      author={Longo, Roberto},
      author={Müger, Michael},
       title={{Multi-Interval Subfactors and Modularity of Representations in
  Conformal Field Theory}},
        date={2001},
     journal={Comm. Math. Phys.},
      volume={219},
       pages={631–669},
      eprint={arXiv:math/9903104},
}

\bib{KaLoXu2005}{article}{
      author={Kac, Victor~G.},
      author={Longo, Roberto},
      author={Xu, Feng},
       title={Solitons in affine and permutation orbifolds},
        date={2005},
        ISSN={0010-3616},
     journal={Comm. Math. Phys.},
      volume={253},
      number={3},
       pages={723\ndash 764},
         url={http://dx.doi.org/10.1007/s00220-004-1160-1},
      review={\MR{2116735 (2006b:81154)}},
}

\bib{KoRu2008}{article}{
      author={Kong, Liang},
      author={Runkel, Ingo},
       title={{Morita classes of algebras in modular tensor categories}},
        date={2008},
        ISSN={0001-8708},
     journal={Adv. Math.},
      volume={219},
      number={5},
       pages={1548–1576},
         url={http://dx.doi.org/10.1016/j.aim.2008.07.004},
      review={\MR{2458146 (2009h:18016)}},
}

\bib{LiNg2014}{incollection}{
      author={Liu, Gongxiang},
      author={Ng, Siu-Hung},
       title={On total {F}robenius-{S}chur indicators},
        date={2014},
   booktitle={Recent advances in representation theory, quantum groups,
  algebraic geometry, and related topics},
      series={Contemp. Math.},
      volume={623},
   publisher={Amer. Math. Soc., Providence, RI},
       pages={193\ndash 213},
         url={http://dx.doi.org/10.1090/conm/623/12462},
      review={\MR{3288628}},
}

\bib{LoXu2004}{article}{
      author={Longo, Roberto},
      author={Xu, Feng},
       title={{Topological sectors and a dichotomy in conformal field theory}},
        date={2004},
        ISSN={0010-3616},
     journal={Comm. Math. Phys.},
      volume={251},
      number={2},
       pages={321–364},
         url={http://dx.doi.org/10.1007/s00220-004-1063-1},
      review={\MR{2100058 (2005i:81087)}},
}

\bib{Mg2002}{article}{
      author={Müger, Michael},
       title={{On the Structure of Modular Categories}},
        date={2002-01},
     journal={ArXiv Mathematics e-prints},
      eprint={arXiv:math/0201017},
}

\bib{Mg2003}{article}{
      author={Müger, Michael},
       title={{From subfactors to categories and topology. {I}. {F}robenius
  algebras in and {M}orita equivalence of tensor categories}},
        date={2003},
        ISSN={0022-4049},
     journal={J. Pure Appl. Algebra},
      volume={180},
      number={1-2},
       pages={81–157},
         url={http://dx.doi.org/10.1016/S0022-4049(02)00247-5},
      review={\MR{1966524 (2004f:18013)}},
}

\bib{Mg2003II}{article}{
      author={Müger, Michael},
       title={{From subfactors to categories and topology. {II}. {T}he quantum
  double of tensor categories and subfactors}},
        date={2003},
        ISSN={0022-4049},
     journal={J. Pure Appl. Algebra},
      volume={180},
      number={1-2},
       pages={159–219},
         url={http://dx.doi.org/10.1016/S0022-4049(02)00248-7},
      review={\MR{1966525 (2004f:18014)}},
}

\bib{Mg2005}{article}{
      author={Müger, Michael},
       title={{Conformal Orbifold Theories and Braided Crossed G-Categories}},
        date={2005},
        ISSN={0010-3616},
     journal={Comm. Math. Phys.},
      volume={260},
       pages={727–762},
         url={http://dx.doi.org/10.1007/s00220-005-1291-z},
}

\bib{MaNg2001}{article}{
      author={Mason, Geoffrey},
      author={Ng, Siu-Hung},
       title={Group cohomology and gauge equivalence of some twisted quantum
  doubles},
        date={2001},
        ISSN={0002-9947},
     journal={Trans. Amer. Math. Soc.},
      volume={353},
      number={9},
       pages={3465\ndash 3509 (electronic)},
         url={http://dx.doi.org/10.1090/S0002-9947-01-02771-4},
      review={\MR{1837244 (2002h:16066)}},
}

\bib{MoSe1989}{article}{
      author={Moore, Gregory},
      author={Seiberg, Nathan},
       title={Classical and quantum conformal field theory},
        date={1989},
        ISSN={0010-3616},
     journal={Comm. Math. Phys.},
      volume={123},
      number={2},
       pages={177\ndash 254},
         url={http://projecteuclid.org/euclid.cmp/1104178762},
      review={\MR{1002038}},
}

\bib{MoSe1990}{incollection}{
      author={Moore, Gregory},
      author={Seiberg, Nathan},
       title={{Lectures on {RCFT}}},
        date={1990},
   booktitle={{Superstrings '89 ({T}rieste, 1989)}},
   publisher={World Sci. Publ., River Edge, NJ},
       pages={1–129},
      review={\MR{1159969 (93m:81133a)}},
}

\bib{NSSFD2008}{article}{
      author={Nayak, Chetan},
      author={Simon, Steven~H.},
      author={Stern, Ady},
      author={Freedman, Michael},
      author={Das~Sarma, Sankar},
       title={Non-abelian anyons and topological quantum computation},
        date={2008},
        ISSN={0034-6861},
     journal={Rev. Modern Phys.},
      volume={80},
      number={3},
       pages={1083\ndash 1159},
  url={https://doi-org.proxy.lib.ohio-state.edu/10.1103/RevModPhys.80.1083},
      review={\MR{2443722}},
}

\bib{Pa2018}{article}{
      author={Passegger, Albert~Georg},
       title={Permutation actions on modular tensor categories of topological
  multilayer phases},
        date={2018},
     journal={arXiv preprint arXiv:1808.06574},
}

\bib{Se2011}{incollection}{
      author={Selinger, P.},
       title={{A survey of graphical languages for monoidal categories}},
        date={2011},
   booktitle={{New structures for physics}},
      series={{Lecture Notes in Phys.}},
      volume={813},
   publisher={Springer},
     address={Heidelberg},
       pages={289–355},
         url={http://dx.doi.org/10.1007/978-3-642-12821-9_4},
      review={\MR{2767048}},
}

\bib{Tu2010}{book}{
      author={Turaev, Vladimir},
       title={Homotopy quantum field theory},
      series={EMS Tracts in Mathematics},
   publisher={European Mathematical Society (EMS), Z\"urich},
        date={2010},
      volume={10},
        ISBN={978-3-03719-086-9},
         url={http://dx.doi.org/10.4171/086},
        note={Appendix 5 by Michael M\"uger and Appendices 6 and 7 by Alexis
  Virelizier},
      review={\MR{2674592}},
}

\bib{Tu1994}{book}{
      author={Turaev, V.G.},
       title={{Quantum Invariants of Knots and 3-Manifolds}},
   publisher={Walter de Gruyter},
        date={1994},
}

\bib{Wang2010}{book}{
      author={Wang, Zhenghan},
       title={Topological quantum computation},
      series={CBMS Regional Conference Series in Mathematics},
   publisher={Published for the Conference Board of the Mathematical Sciences,
  Washington, DC; by the American Mathematical Society, Providence, RI},
        date={2010},
      volume={112},
        ISBN={978-0-8218-4930-9},
         url={https://doi.org/10.1090/cbms/112},
      review={\MR{2640343}},
}

\end{biblist}
\end{bibdiv}

\newpage
\begin{appendix}

\begin{section}{Fusion rules for \texorpdfstring{$\mathbbm{Z}/4\mathbbm{Z}$}{Z/4Z} permutation extensions of Fib}
    Let $\cC$ be the Fibonacci category with $\Irr(\cC)=\{\tu,\tau\}$ with $\tau\otimes \tau \cong \tu\oplus \tau$ and $\cD=\cC\wr \ZZ/4\ZZ$. Then
    $\Irr(\cD_i)=\{(i,\tu),(i,\tau)\}$ for $i=1,3$ and
    $\Irr(\cD_2)=\{(2,\tu\tu),(2,\tu\tau),(2,\tau\tu),(2,\tau\tau)\}$ and
  \begin{align*}
        (i,\tu)(i+2,\tu)
        &=
        \tu\tu\tu\tu+\tu\tu\tau\tau+\cdots+
        \tu\tau\tau\tau+\cdots + 2\tau\tau\tau\tau\\
        (i,\tu)(i+2,\tau)&
        =\tu\tu\tu\tau+\cdots+
        \tu\tu\tau\tau+\cdots+ 
        2\tu\tau\tau\tau+\cdots+
        3\tau\tau\tau\tau
        \\
        (i,\tau)(i+2,\tau)&
        =\tu\tu\tu\tu+ \tu\tu\tu\tau+\cdots+
        2\tu\tu\tau\tau+\cdots+ 
        3\tu\tau\tau\tau+\cdots+
        5\tau\tau\tau\tau
        \\
        (i,\tu)(i,\tu) &=
        2 (2,\tu\tu)
        +(2,\tu\tau)+(2,\tau\tu) 
        + 3(2,\tau\tau)
        \\
        (i,\tu)(i,\tau) &=
         3(2,\tu\tu)
        +3(2,\tu\tau)+(2,\tau\tu) 
        + 4(2,\tau\tau)
        \\
        (i,\tau)(i,\tau) 
        &=         
        3(2,\tu\tu)+4(2,\tu\tau)+4(2,\tau\tu)+7(2,\tau\tau)
        \\
        \tu\tu\tu\tau(i,\tu)&=(i,\tau)\\
        \tu\tu\tau\tau(i,\tu)&=(i,\tu)+(i,\tau)\\
        \tu\tau\tau\tau(i,\tu)&=(i,\tu)+2(i,\tau)\\
        \tau\tau\tau\tau(i,\tu)&=2(i,\tu)+3(i,\tau)
        \\
        \tu\tu\tu\tau(i,\tau)&=(i,\tu)+(i,\tau)\\
        \tu\tu\tau\tau(i,\tau)&=(i,\tu)+2(i,\tau)\\
        \tu\tau\tau\tau(i,\tau)&=2(i,\tu)+3(i,\tau)\\
        \tau\tau\tau\tau(i,\tau)&=3(i,\tu)+5(i,\tau)
        \\
        \tu\tu\tu\tau(2,\tu\tu)&=(2,\tu\tau)\\
        \tu\tu\tau\tau(2,\tu\tu)&=(2,\tau\tau)\\
        \tu\tau\tu\tau(2,\tu\tu)&=(2,\tu\tu)+(2,\tu\tau)\\
        \tu\tau\tau\tau(2,\tu\tu)&=(2,\tau\tu)+(2,\tau\tau)\\
        \tau\tau\tau\tau(2,\tu\tu)&=(2,\tu\tu)+(2,\tu\tau)+(2,\tau\tu)+(2,\tau\tau)
        \\
        \tu\tu\tu\tau(2,\tu\tau)&=(2,\tu\tu)+(2,\tu\tau)\\
        \tu\tu\tau\tu(2,\tu\tau)&=(2,\tau\tu)+(2,\tau\tau)\\
        \tu\tu\tau\tau(2,\tu\tau)&=(2,\tau\tu)+(2,\tau\tau)\\
        \tu\tau\tu\tau(2,\tu\tau)&=(2,\tu\tu)+2(2,\tu\tau)\\
        \tu\tau\tau\tau(2,\tu\tau)&=(2,\tau\tu)+2(2,\tau\tau)\\
        \tau\tau\tau\tau(2,\tu\tau)&=(2,\tu\tu)+2(2,\tu\tau)+(2,\tau\tu)+2(2,\tau\tau)
       \\
        \tu\tu\tu\tau(2,\tau\tau)&=(2,\tau\tu)+(2,\tau\tau)\\
        \tu\tu\tau\tau(2,\tau\tau)&=(2,\tu\tu)+(2,\tu\tau)+(2,\tau\tu)+(2,\tau\tau)\\
        \tu\tau\tu\tau(2,\tau\tau)&=(2,\tau\tu)+2(2,\tau\tau)\\
        \tu\tau\tau\tau(2,\tau\tau)&=(2,\tu\tu)+2(2,\tu\tau)+(2,\tau\tu)+2(2,\tau\tau)\\
        \tau\tau\tau\tau(2,\tau\tau)&=(2,\tu\tu)+2(2,\tu\tau)+2(2,\tau\tu)+4(2,\tau\tau)
        \end{align*}
    \begin{align*}    
        (i,\tu)(2,\tu\tu)&=2(i+2,\tu)+(i+2,\tau)\\
        (i,\tu)(2,\tu\tau)&=(i+2,\tu)+3(i+2,\tau)\\
        (i,\tu)(2,\tau\tau)&=3(i+2,\tu)+4(i+2,\tau)\\
        (i,\tau)(2,\tu\tu)&=(i+2,\tu)+3(i+2,\tau)\\
        (i,\tau)(2,\tu\tau)&=3(i+2,\tau)+4(i+2,\tau)\\
        (i,\tau)(2,\tau\tau)&=4(i+2,\tau)+7(i+2,\tau)\\
        (2,\tu\tu)(2,\tu\tu)&=\tu\tu\tu\tu+ \tau\tu\tau\tu+\tu\tau\tu\tau+\tau\tau\tau\tau\\
        (2,\tu\tu)(2,\tu\tau)&=\tu\tu\tu\tau+\tu\tau\tu\tu+\tu\tau\tu\tau+
        \tau\tau\tau\tu +\tau\tu\tau\tau+\tau\tau\tau\tau
        \\
        (2,\tu\tu)(2,\tau\tau)&=\tu\tu\tu\tu+
        \tu\tu\tu\tau+\cdots+\tu\tu\tau\tau+\cdots+
        \tu\tau\tau\tau+\cdots+\tau\tau\tau\tau\\
        (2,\tu\tau)(2,\tu\tau)&=\tu\tu\tu\tu+
        \tu\tu\tu\tau+\tu\tau\tu\tu+
        2\tu\tau\tu\tau+\tau\tu\tau\tu+\tau\tu\tau\tau+\tau\tau\tau\tu + 2\tau\tau\tau\tau\\
        (2,\tu\tau)(2,\tau\tau)&=
        \tu\tu\tau\tu+\tau\tu\tu\tu+\tu\tu\tau\tau+\tu\tau\tau\tu +\tau\tu\tu\tau
        +\tau\tu\tau\tu+\tau\tau\tu\tu+
        \\&+ 2\tu\tau\tau\tau +2\tau\tau\tu\tau+
        \tau\tu\tau\tau+\tau\tau\tau\tu+2\tau\tau\tau\tau\\
        (2,\tau\tau)(2,\tau\tau)&=\tu\tu\tu\tu
        +\tu\tu\tu\tau+\cdots + \tu\tu\tau\tau+
        2\tu\tau\tu\tau+\cdots + 2\tu\tau\tau\tau+\cdots+
        4\tau\tau\tau\tau
    \end{align*}
    where we write ``$\cdots$'' for obvious permutation of objects.
    \begin{table}
    \begin{tabular}{l|rrrrrrrr}
         $n$&0&1&2&3&4&5&6&7  \\
         \hline
         $N_{\tau^n}$ &1&0&1&1&2&3&5&8\\
         ${}^1N_{\tau^n}$&2&1&3&4&7&11&18&29
    \end{tabular}
    \caption{The fusion coefficients 
    ${}^gN_{\tau^n}$ are given by the Fibonacci and Lucas numbers for $g=0,1$, respectively.}
    \end{table} 

\end{section}

\end{appendix}

\end{document}